\documentclass[11pt]{amsart}

\usepackage{fullpage}
\usepackage{booktabs}
\usepackage{float}
\usepackage[hyperfootnotes=false, colorlinks, linkcolor={blue}, citecolor={magenta}, filecolor={blue}, urlcolor={blue}, plainpages=false, pdfpagelabels]{hyperref}
\usepackage{amsmath, amscd, amsthm, amssymb}
\usepackage{ae, aecompl}
\usepackage[utf8]{inputenc}
\usepackage{enumitem}
\usepackage{csquotes}
\usepackage[T1]{fontenc}
\usepackage{rotating}
\usepackage{longtable}
\usepackage[english]{babel}
\usepackage{xcolor}
\usepackage{eucal}
\usepackage{graphicx}

\usepackage{amsthm}

\usepackage{stmaryrd}

\usepackage{appendix}

\usepackage{scrextend}
\usepackage{mathtools}
\usepackage[numbers,sort&compress]{natbib}
\usepackage{etoolbox}

\usepackage{setspace}
\def\sl{{\mathfrak{sl}}}
\newtheorem{theorem}{Theorem}[section]

\newtheorem{lemma}[theorem]{Lemma}

\newtheorem{proposition}[theorem]{Proposition}

\newtheorem{corollary}[theorem]{Corollary}

\theoremstyle{definition}

\theoremstyle{remark}
\newtheorem*{remark}{Remark}

\numberwithin{equation}{section}
\newcommand{\mm}[4]{\left(\begin{smallmatrix} #1 & #2\\ #3 & #4\end{smallmatrix}\right)}
\newcommand{\mb}[4]{\left(\begin{array}{cc} #1 & #2\\ #3 & #4\end{array}\right)}

\DeclareMathOperator{\val}{val}
\DeclareMathOperator{\tr}{tr}

\DeclareMathOperator{\SO}{SO}
\DeclareMathOperator{\Spin}{Spin}

\DeclareMathOperator{\SU}{SU}

\DeclareMathOperator{\SL}{SL}
\DeclareMathOperator{\GL}{GL}

\DeclareMathOperator{\diag}{diag}
\DeclareMathOperator{\charf}{char}

\newcommand{\F}{\mathbb{F}}
\newcommand{\Z}{\mathbb{Z}}
\newcommand{\Q}{\mathbb{Q}}
\newcommand{\A}{\mathbb{A}}

\newcommand{\C}{\mathbb{C}}

\newcommand{\R}{\mathbb{R}}

\definecolor{carmine}{rgb}{0.59, 0.0, 0.09}
\definecolor{fatma}{rgb}{0.04706, 1.33725, 0.56667}

\begin{document}
\title{The completed standard $L$-function of modular forms on $G_2$}
\author{Fatma \c{C}\.{i}\c{c}ek}
\address{Department of Mathematics and Statistics\\ University of Northern British Columbia\\ Prince George, BC, Canada}
\email{cicek.ftm@gmail.com}
\author{Giuliana Davidoff}
\address{Department of Mathematics and Statistics\\Mount Holyoke College\\South Hadley, MA, USA}
\email{gdavidof@mtholyoke.edu}
\author{Sarah Dijols}
\address{Department of Mathematics\\ University of British Columbia\\ Vancouver, BC, Canada}

\email{sarah.dijols@hotmail.fr}
\author{Trajan Hammonds}
\address{Department of Mathematics\\ Princeton University \\ Princeton, NJ, USA}
\email{trajanh@princeton.edu}
\author{Aaron Pollack}
\address{Department of Mathematics\\ University of California San Diego\\ La Jolla, CA, USA}
\email{apollack@ucsd.edu}
\thanks{A. Pollack has been supported by the Simons Foundation via Collaboration Grant number 585147 and by the NSF via award number 2101888.}
\author{Manami Roy}
\address{Department of Mathematics\\ Fordham University \\ Bronx, NY, USA}
\email{mroy17@fordham.edu}
\subjclass{11F70, 11F66, 11S40, 11F30}
\keywords{Rankin-Selberg integral, functional equation, modular forms on $G_2$, trivial zeroes}

\begin{abstract}
The goal of this paper is to provide a complete and refined study of the standard $L$-functions $L(\pi,\operatorname{Std},s)$ for certain non-generic cuspidal automorphic representations $\pi$ of $G_2(\A)$. For a cuspidal automorphic representation $\pi$ of $G_2(\A)$ that corresponds to a modular form $\varphi$ of level one and of even weight on $G_2$, we explicitly define the completed standard $L$-function, $\Lambda(\pi,\operatorname{Std},s)$. Assuming that a certain Fourier coefficient of $\varphi$ is nonzero, we prove the functional equation $\Lambda(\pi,\operatorname{Std},s) = \Lambda(\pi,\operatorname{Std},1-s)$. Our proof proceeds via a careful analysis of a Rankin-Selberg integral that is due to an earlier work of Gurevich and Segal.
\end{abstract}

\maketitle

\setcounter{tocdepth}{1}
\tableofcontents


\section{Introduction}
\subsection{History}
Let $G_2$ denote the split exceptional linear algebraic group over $\Q$ of Dynkin type $G_2$, and suppose that $\pi$ is a cuspidal automorphic representation of $G_2(\A)$. The study of $L$-functions associated to such representations has a substantial history. Piatetski--Shapiro, Rallis and Schiffmann \cite{PSRallisSchiffmann} were the first to study such an $L$-function by constructing a Rankin-Selberg integral for the tensor product $L$-function of $\pi$ and a cuspidal automorphic representation on $\GL_2.$ Their result applies to the $\pi$ that are globally generic, that is, those $\pi$ that admit a nonvanishing Whittaker coefficient. Later Ginzburg \cite{Ginzburg1993} proved that for generic $\pi$, the partial standard $L$-function $L^S(\pi,\operatorname{Std},s)$ has a meromorphic continuation with at most a simple pole by constructing an appropriate Rankin-Selberg integral.

For cuspidal representations $\pi$ that are not necessarily generic, the corresponding Rankin-Selberg integrals were constructed in the works of Ginzburg and Hundley \cite{GinzburgHundley2015}, Gurevich and Segal \cite{GurevichSegal} and then Segal \cite{Segal2017}. It was proven in \cite{Segal2017} that the partial standard $L$-function of such a representation $\pi$ admits a meromorphic continuation to the complex plane. However, bounding the poles of the $L$-function $L^S(\pi,\operatorname{Std},s)$ in a left half-plane, and proving a functional equation relating its values at $s$ to its values at $(1-s)$ are difficult problems. This is, in part, due to the difficulty of analyzing local $L$-functions and local zeta integrals at the ramified finite places and at the archimedean place.


\subsection{Statements of results}
Modular forms on $G_2$ were introduced by Gan, Gross and Savin in \cite{GanGrossSavin}. Briefly, these are automorphic forms on $G_2(\A)$ that correspond to representations in the form $\pi = \pi_f \otimes \pi_\infty$ where $\pi_f$ denotes a representation that is unramified at every finite place and $\pi_\infty$ is a certain quaternionic discrete series representation of $G_2(\R).$ Let $K$ denote a maximal compact subgroup of $G_2(\R)$, so that $K \simeq (\SU(2) \times \SU(2))/\{\pm 1\}$ with the first copy of $\SU(2)$ being the long root and the second being the short root. Then for $\ell \geq 2$, there is a discrete series representation $\pi_{\ell,\infty}$ of $G_2(\R)$ whose minimal $K$-type is $\textrm{Sym}^{2\ell}(\C^2) \boxtimes \mathbf{1}$ as a representation of $\SU(2) \times \SU(2)$. Such representations $\pi_{\ell,\infty}$ are not generic.

Let $\ell \geq 2$ be an even integer. We define the archimedean $L$-factor as
\begin{equation}\label{defn: arch factor}
L_\infty(\pi_{\ell,\infty},s) = \Gamma_\C(s+\ell-1) \Gamma_\C(s+\ell)\Gamma_\C(s+2\ell-1)\Gamma_\R(s+1).
\end{equation}
Here
\[
\Gamma_\R(s) = \pi^{-s/2}\Gamma(s/2) \quad \text{and}  \quad \Gamma_\C(s) = 2 (2\pi)^{-s}\Gamma(s),
\]
where $\Gamma$ is the usual gamma function. It is worthwhile to point out that Gross and Savin \cite[p. 168]{GrossSavin} had previously defined the archimedean $L$-factor for representations of the compact group $G_2^{\textrm{ c}}(\R)$. We easily see that our archimedean $L$-factor agrees with theirs by setting $k_1=0$ and $k_2 = \ell-2$ in their notation.

For such representations $\pi = \pi_f \otimes \pi_{\ell,\infty}$, an $L$-function $L(\pi,\operatorname{Std},s)$ is defined. Then the completed $L$-function is given by
\[
\Lambda(\pi,\operatorname{Std},s) = L_\infty(\pi_{\ell,\infty},s)L(\pi,\operatorname{Std},s),
\]
where $L_\infty(\pi_{\ell,\infty},s)$ is as given in \eqref{defn: arch factor}.
Our main results concern these $L$-functions.

Further, we recall that such a representation $\pi$ has an associated cuspidal modular form $\varphi_{\pi}$ on $G_2$ that has weight $\ell$ and level one. It was proved in \cite{GanGrossSavin} that the Fourier coefficients of such a modular form $\varphi_{\pi}$ depend on a cubic ring $T$ such that $T \otimes \R \simeq \R \times \R \times \R$. We thus denote them by $a_{\varphi_{\pi}}(T)$.

Our first result is as follows.

\begin{theorem} \label{thm:intro}
Suppose that $\varphi$ is a level one cuspidal modular form on $G_2$ of positive even weight $\ell$ that generates the cuspidal automorphic representation $\pi.$ Further, assume that the Fourier coefficient of $\varphi$ corresponding to the split cubic ring $\Z \times \Z \times \Z$ is nonzero. Then
$$
\Lambda(\pi,\operatorname{Std},s) = \Lambda(\pi,\operatorname{Std},1-s)
$$
for all $s\in \mathbb C$.
\end{theorem}

We carefully note that at present, it is not known whether there exists such a level one, even weight cuspidal modular form $\varphi$ with a nonzero $\Z \times \Z \times \Z$ Fourier coefficient. However, it is certainly expected (by analogy with Siegel modular forms of genus two) that these modular forms exist in abundance.  We also note that the recent work \cite{dalal} provides a dimension formula for the space of level one, cuspidal modular forms on $G_2$ of weight at least three. However, \cite{dalal} sheds no light about the existence of specific nonzero Fourier coefficients of such forms, as is needed in Theorem \ref{thm:intro}.

The proof of Theorem \ref{thm:intro} is based on a refined analysis of a Rankin-Selberg integral that was defined in \cite{GurevichSegal}. Moreover, a Dirichlet series for the $L$-function $L(\pi,\operatorname{Std},s)$ follows from the proof of Theorem \ref{thm:intro}. We also have

\begin{corollary}\label{cor:introDS}
Let the assumptions be as in Theorem \ref{thm:intro}, and let $a_\varphi(T)$ denote the Fourier coefficient of $\varphi$ corresponding to the cubic ring $T$ that satisfies $T \otimes \R \simeq \R \times \R \times \R.$ Then
	\[
	\sum_{T \subseteq \Z^3, n \geq 1}{\frac{a_\varphi(\Z + n T)}{[\Z^3:T]^{s-\ell+1}n^{s}}} = a_{\varphi}(\Z^3) \frac{L(\pi,\operatorname{Std},s-2\ell+1)}{\zeta(s-2\ell+2)^2\zeta(2s-4\ell+2)}.
	\]
Here the sum is over the subrings $T$ of $\Z \times \Z \times \Z$ and integers $n \geq 1$.
\end{corollary}

In~\cite{PollackG2}, Pollack gave a streamlined account of the Rankin-Selberg integrals in \cite{GurevichSegal} and \cite{Segal2017} whereby simplifying some of their computations. He used his analysis of the Rankin-Selberg integral to provide a Dirichlet series representation for the standard $L$-function of modular forms on $G_2$ outside the primes $p=2$ and $p=3$, and began some calculations of the archimedean zeta integral associated to the global Rankin-Selberg convolution. Thus, Theorem \ref{thm:intro} and Corollary \ref{cor:introDS} bring the work that began in \cite{PollackG2} to completion.


Observe that by studying the archimedean factor $L_\infty(\pi_{\ell,\infty},s),$ one can verify that the integers $1,3,5,\ldots,\ell-1$ are critical for $L(\pi,\operatorname{Std},s)$ in the sense of Deligne, that is, both of the values $L_\infty(\pi_{\ell,\infty},s)$ and $L_\infty(\pi_{\ell,\infty},1-s)$ are finite at these integers. It would be extremely interesting to obtain a special value result in the direction of Deligne's conjecture for these $L$-values. While such a result is beyond the reach of our methods, we can obtain a result on what can be considered the most basic special value, namely, a result on the trivial zeros of the $L$-function $L(\pi,\operatorname{Std},s)$. This is an immediate corollary of the functional equation of the completed $L$-function. In more detail, the completed $L$-function $\Lambda(\pi,\operatorname{Std},s)$ is finite and nonzero for Re($s)\gg 0$, and also for Re($s)\ll 0$ by using the functional equation. However, the archimedean factor $L_\infty(\pi_{\ell,\infty},s)$ has poles at negative integers of sufficiently large absolute value. These poles are compensated for by the zeros of the standard $L$-function. We therefore deduce the following from Theorem \ref{thm:intro}.

\begin{corollary}
Let the assumptions be as in Theorem \ref{thm:intro}. Then $L(\pi,\operatorname{Std},s)$ vanishes to order $3$ at negative even integers of sufficiently large absolute value, and vanishes to order $4$ at negative odd integers of sufficiently large absolute value.
\end{corollary}


\subsection{Outline of the proof of Theorem \ref{thm:intro}}
As mentioned, our proof of Theorem \ref{thm:intro} is based on a refined analysis of the Rankin-Selberg integral in \cite{GurevichSegal} and is a continuation of the work in \cite{PollackG2}.

Let $G$ denote the split group $\Spin(8)$. If $\varphi$ is a modular form on $G_2$ of weight $\ell$, then by definition, $\varphi$ is a $\mathbf{V}_{\ell}$-valued automorphic function on $G_2(\A)$, where $\mathbf{V}_{\ell} = \textrm{Sym}^{2\ell}(\C^2)$ (see \cite{PollackG2} for a more detailed account of modular forms on $G_2$).  For a normalized Eisenstein series $E_{\ell}^*(g,s)$ on $G(\A)$ that takes values in $\mathbf{V}_{\ell}$, we will consider the Rankin-Selberg integral
\[
I_\ell(\varphi,s) = \int_{G_2(\Q)\backslash G_2(\A)}{\{\varphi(g),E_{\ell}^*(g,s)\}_K\,dg}.
\]
Here $\{\cdot, \cdot\}_K: \mathbf{V}_{\ell} \otimes \mathbf{V}_{\ell} \rightarrow \C$ is a $K$-equivariant pairing. In order to obtain Theorem \ref{thm:intro}, we will prove that
	\[
	I_\ell(\varphi,s)  =a_{\varphi}(\Z^3) \Lambda(\pi,\operatorname{Std},s-2)
	\]
up to a nonzero constant, and that the Eisenstein series $E_\ell^*(g,s)$ satisfies the functional equation
\[
E_{\ell}^*(g,s) = E_{\ell}^*(g,5-s).
\]

For the proof of the first statement, we will analyze local integrals $I_p(s)$ for finite primes $p$, which will be defined in \eqref{Ips}, and an archimedean integral $I^*(s;\ell)$ defined in \eqref{eq:norm zeta integral}. We will prove that these local integrals are equal to the corresponding local $L$-factors up to some simple factors. For $p \geq 5$, the local integrals $I_p(s)$ were analyzed in \cite{GurevichSegal}. To carry out the computation of these integrals for $p =2,3$, we follow a method in \cite{PollackG2} and use some results on cubic rings. The analysis of the integral $I^*(s;\ell)$ was begun in \cite{PollackG2}, where the computation was reduced to that of an integral $J'(s)$ over the space of real binary cubics of a general form that was previously considered by Shintani \cite{Shintani}. We will evaluate the integral $J'(s)$ explicitly in terms of the gamma function, thereby proving that $I^*(s;\ell) = L(\pi_{\ell,\infty},s-2)$ up to a nonzero constant.

To prove the functional equation for $E_{\ell}^*(g,s)$, we will use Langlands' functional equation for the Eisenstein series. Since our Eisenstein series $E_{\ell}^*(g,s)$ is not spherical at the archimedean place, we will make a careful analysis of certain archimedean intertwining operators.

\begin{remark}
The methods in this paper are somewhat flexible but also have some limitations.

For instance, in terms of the calculations of the unramified integrals at the finite places, our restriction to $\mathbb{Q}$ is simply for convenience. These calculations would follow just as easily over other ground fields. However, where we really use our assumption that $\mathbb{Q}$ is the ground field is in the archimedean calculation. Indeed, we do not expect that the calculations in Sections \ref{sec:Eis} and \ref{sec:archInt} will have close analogues for other number fields.

  Also, notice that we have several assumptions on our modular forms on $G_2.$ Firstly, we only consider modular forms of level one. This allows us to do an unramified computation at \emph{every} finite place. Without this assumption, it would also be difficult to obtain a precise functional equation for an Eisenstein series that is used in the Rankin-Selberg integral. We also restrict our discussion to quaternionic modular forms due to reasons that can be considered to be  \emph{purely archimedean}. For such modular forms, we can define a completed $L$-function, prove its relation to the global Rankin-Selberg convolution, and prove its functional equation. This archimedean assumption does not make the unramified computations any easier or different. However, by working with quaternionic modular forms, we are able to compute the archimedean integral as in Section \ref{sec:archInt}. We would not expect to be able to do an analogous computation for non-spherical, non-quaternionic automorphic forms.

    Finally, note that we only consider the $\Z \times \Z \times \Z$ Fourier coefficient of quaternionic modular forms. This assumption is used to make the unramified calculation as simple as possible at \emph{every} finite place.  In particular, replacing $\Z \times \Z \times \Z$ with another maximal totally real cubic ring, we would still expect to be able to do the resulting archimedean calculation. Importantly, Theorem \ref{thm:intro} and its corollaries require that our modular form supports a nonzero Fourier coefficient corresponding to the split cubic ring $\Z \times \Z \times \Z$.
 \end{remark}


\subsection{Organization of the paper}
We now provide an outline of our paper. In Section \ref{sec:Setups}, we setup some notation. Then in the next section, we give an overview of the Rankin-Selberg integral and present our strategy to calculate the non-archimedean local integrals. In Section \ref{Sect. binary cubic}, we prove some results that relate some cubic rings to some binary cubic forms. In Section \ref{sec:ABF}, we compute the Fourier coefficient of the so-called approximate basic function. In Section \ref{sec:non-arch}, we complete the computation for the case of unramified primes. This involves computing the function  $\Phi_{p,\chi}(t,g)$, which is defined in \eqref{Phichi} and is related to the inducing section of the Eisenstein series, and some calculations with certain Hecke operators. In Section \ref{sec:Eis}, we prove the functional equation of the Eisenstein series $E_\ell^*(g,s)$ and then in Section \ref{sec:archInt}, we compute the archimedean zeta integral $I^*(s;\ell)$. Finally in Section 9, we combine our work and complete the proof of our results.


\subsection{Acknowledgments}
This paper is an outgrowth of the workshop Rethinking Number Theory: 2020 that was organized by Heidi Goodson, Christelle Vincent and Mckenzie West. The authors extend their sincere thanks to the organizers of the workshop, without which this paper would not have been written. We also thank the anonymous referee for reading our manuscript very carefully and providing many valuable comments and suggestions.


\section{Setups}\label{sec:Setups}
\subsection{Octonions and reductive groups}
It is well-known that $G_2$ is defined as the automorphism group of an octonion algebra. In this section, we use the split octonions algebra $\Theta$ in the Zorn model (see \cite[Section 2.1]{PollackG2}) to view $G_2$ as a subgroup of Spin$(\Theta)$. We thus begin with a review of some of the notation used in \cite{PollackG2}. The standard representation $V_3$ of $\SL_3$ and its dual representation  $V_3^\vee$ will be fixed. The space $V_3$ has a standard basis $\{e_1, e_2, e_3\}$ and $V_3^\vee$ has the dual basis $\{e_1^*, e_2^*, e_3^*\}$.  Note that for each $j=1, 2, 3$, we make the identifications
\[
e_j \in V_3 \leftrightarrow
\bigg(\begin{array}{cc} {0}&{e_j}\\{0}&{0} \end{array}\bigg) \in \Theta \quad
\text{and}
\quad
e_j^* \in V_3^\vee \leftrightarrow \bigg(\begin{array}{cc} {0}&{0}\\ {e_j^*}&{0}\end{array}\bigg) \in \Theta.
\]

Using the quadratic norm on $\Theta$, we can define the group $G'=\SO(\Theta)$. Now, let $G$ denote the algebraic group $\Spin(\Theta)$ defined as
\[
G=\big\{(g_1,g_2,g_3) \in \SO(\Theta)^3 : (g_1 x_1, g_2 x_2, g_3 x_3) = (x_1, x_2,x_3) \,\, \text{ for all }\, x_1,x_2, x_3 \in \Theta\big\},
\]
where $(x_1,x_2,x_3) = \tr_\Theta(x_1 (x_2 x_3))$. We fix a map $\textrm{proj}_1: G \rightarrow G'$ as $(g_1,g_2,g_3) \mapsto g_1$. This map induces an isomorphism on Lie algebras.
\vspace{0,2cm}

Let $\Theta_0$ be the standard maximal lattice in the Zorn model. Then $\Theta_0$ consists of the matrices $\mm{a}{v}{\phi}{d}$ where $a, d \in \Z$, $v$ is in the $\Z$-span of the $\{e_1, e_2, e_3\}$ and $\phi$ is in the $\Z$-span of $\{e_1^*, e_2^*, e_3^*\}$.

We set $K_f = \prod_{p}{K_p}$, where $K_p$ is the hyperspecial maximal compact subgroup of $G(\Q_p)$ that is specified as the stabilizer of $(\Theta_0 \otimes \Z_p)^3$ inside of $G(\Q_p)$. That $K_p$ is hyperspecial follows from \cite[Section 4]{grossGpsZ} and \cite[Proposition 5.4]{Conrad}. Similarly, we let $G_2(\Z_p)$ be the stabilizer of $\Theta_0 \otimes \Z_p$ inside $G_2(\Q_p)$. This is a hyperspecial maximal compact subgroup.
\vspace{0,2cm}

We now define a maximal compact subgroup of $G'(\R)$, $K_\infty'$, as follows. Given $v \in V_3$, suppose that $\widetilde{v} \in V_3^\vee$ is given by the linear mapping $e_j \mapsto e_j^*$ on $V_3$. Similarly, for $\phi \in V_3^\vee$, let $\widetilde{\phi}$ in $V_3$ be given by the linear mapping $e_j^* \mapsto e_j$. We also define a quadratic form $q_{\textrm{maj}}$ on $\Theta \otimes \R$ by
\[
q_{\textrm{maj}}\bigg(
\bigg(
\begin{array}{cc}
a & v \\ \phi & d
\end{array}
\bigg)
\bigg)
= a^2 + d^2 + (v,\widetilde{v}) + (\widetilde{\phi},\phi),
\]
where $(\,,)$ denotes the evaluation pairing between $V_3$ and $V_3^\vee$. Then $K_\infty'$ is defined as the subgroup of $G'(\R)$ that preserves the quadratic form $q_{\textrm{maj}}$. We now let $K_\infty \subseteq G(\R)$ be the inverse image of $K_\infty'$ under the map $\textrm{proj}_1: G \rightarrow G'$.

Put another way, we define $\iota: \Theta \rightarrow \Theta$ as
\[
\iota\bigg(
\bigg(
\begin{array}{cc}
a & v \\ \phi & d
\end{array}
\bigg)
\bigg) =
\bigg(
\begin{array}{cc}
d & -\widetilde{\phi} \\ -\widetilde{v} & a
\end{array}
\bigg).
\]
If $x = \mm{a}{v}{\phi}{d}$, then $q_{\textrm{maj}}(x) = (x, \iota(x))$.  Conjugation by $\iota$ induces a Cartan involution on $G'$ and on $G_2 \subseteq G$ (see \cite[Claim 2.1]{PollackG2}).


\subsection{Lie algebra definitions}
The maps $G_2 \rightarrow G \rightarrow G'$ induce $\textrm{Lie}(G_2) \rightarrow \textrm{Lie}(G') \simeq \wedge^2 \Theta$. This embedding is the one specified in Section 2.2 in \cite{PollackG2}, and we will use notation from that section.

The Heisenberg parabolic $P_{G}$ of $G$ is defined to be the one which stabilizes the line spanned by $E_{13} = e_3^* \wedge e_1$ in $\wedge^2 \Theta$. The Heisenberg parabolic $P$ of $G_2$ is similarly defined as the stabilizer of the line spanned by $E_{13}$ in $\textrm{Lie}(G_2)$, thus $P_{G} \cap G_2 = P$.

\subsection{Setup for $K_\infty$}
Let
\[
K_\infty' = S(O(4)\times O(4)) = \big\{(g_1,g_2) \in O(4) \times O(4): \det(g_1)\det(g_2) = 1\big\}.
\]
We remind the reader that $\SO(4) = (\SU(2) \times \SU(2))/\{\pm 1\}$. Thus there are four copies of $\sl_2$ in $\textrm{Lie}(K_\infty) \otimes \C$. We will introduce them explicitly as these $\sl_2$'s will be used in Section \ref{sec:Eis}. Note that $\textrm{Lie}(K_\infty) \subseteq \textrm{Lie}(G') \simeq \wedge^2 \Theta$. Let
\begin{equation}\label{eqn:bE}
\big\{b_1,b_2,b_3,b_4,b_{-4},b_{-3},b_{-2},b_{-1}\big\} = \big\{e_1,e_3^*,\epsilon_2,e_2^*,-e_2,\epsilon_1,-e_3,-e_1^*\big\}
\end{equation}
in order. Here $\epsilon_1 = \mm{1}{0}{0}{0}$ and $\epsilon_2 = \mm{0}{0}{0}{1}$. With the basis $\{b_1,b_2,b_3,b_4,b_{-4},b_{-3},b_{-2},b_{-1}\}$ of $\Theta$, one has $(b_j,b_k) = (b_{-j},b_{-k}) = 0$ and $(b_j, b_{-k}) = \delta_{jk}$. The involution $\iota$ satisfies $\iota(b_j) = b_{-j}$ and $\iota(b_{-j}) = b_j$ for $j =1,2,3,4$. Define
\begin{align*}
u_1 &= \frac{1}{\sqrt{2}}\big(b_1 + b_{-1}\big), \quad u_2 =  \frac{1}{\sqrt{2}}\big(b_2 + b_{-2}\big), \\
v_1 &=  \frac{1}{\sqrt{2}}\big(b_3 + b_{-3}\big), \quad v_2 =  \frac{1}{\sqrt{2}}\big(b_4 + b_{-4}\big),
\end{align*}
and
\begin{align*}
u_{-1} &= \frac{1}{\sqrt{2}}\big(b_1 - b_{-1}\big), \quad u_{-2} =  \frac{1}{\sqrt{2}}\big(b_2 - b_{-2}\big), \\
v_{-1} &=  \frac{1}{\sqrt{2}}\big(b_3 - b_{-3}\big), \quad v_{-2} =  \frac{1}{\sqrt{2}}\big(b_4 - b_{-4}\big). \end{align*}
With the above notation, we now specify the four copies of $\sl_2$ in $\textrm{Lie}(K_\infty) \otimes \C$. One copy of $\sl_2$ has the basis consisting of
\begin{itemize}
\setlength\itemsep{0.3em}
	\item $\displaystyle e^+ = \frac{1}{2}(u_1-iu_2) \wedge (v_1 - iv_2),$
	\item $\displaystyle h^+ = i (u_1 \wedge u_2 + v_1 \wedge v_2),$
	\item $\displaystyle f^+ = -\frac{1}{2} (u_1+iu_2) \wedge (v_1 + iv_2)$.
\end{itemize}
The other $\sl_2$ from the first $\SO(4)$ in $K_\infty' = S(O(4)\times O(4))$ is obtained by replacing $v_2$ with $-v_2$ in the above formulas. Thus it has a basis that consists of
\begin{itemize}
    \setlength\itemsep{0.3em}
	\item $\displaystyle e'^+ = \frac{1}{2}(u_1-iu_2) \wedge (v_1 + iv_2),$
	\item $\displaystyle h'^+ = i (u_1 \wedge u_2 - v_1 \wedge v_2),$
	\item $\displaystyle f'^+ = -\frac{1}{2} (u_1+iu_2) \wedge (v_1 - iv_2)$.
\end{itemize}
The third copy of $\sl_2$ has the basis consisting of
\begin{itemize}
	\setlength\itemsep{0.3em}
	\item $\displaystyle e^- = \frac{1}{2}(u_{-1}-iu_{-2}) \wedge (v_{-1} - iv_{-2}),$
	\item $\displaystyle h^- = -i (u_{-1} \wedge u_{-2} + v_{-1} \wedge v_{-2}),$
	\item $\displaystyle f^- = -\frac{1}{2} (u_{-1}+iu_{-2}) \wedge (v_{-1} + iv_{-2})$.
\end{itemize}
Finally, the basis of the fourth copy of $\sl_2$ consists of
\begin{itemize}
	\setlength\itemsep{0.3em}
	\item $\displaystyle e'^- = \frac{1}{2}(u_{-1}-iu_{-2}) \wedge (v_{-1} + iv_{-2}),$
	\item $\displaystyle h'^- = -i (u_{-1} \wedge u_{-2} - v_{-1} \wedge v_{-2}),$
	\item $\displaystyle f'^- = -\frac{1}{2} (u_{-1}+iu_{-2}) \wedge (v_{-1} - iv_{-2})$.
\end{itemize}
The compatible Cartan involutions on $G_2$ and $G$, and the embedding $G_2 \subseteq G$ picks out a distinguished $\sl_2$ of the above four, the image of the long root $\sl_2$ of $G_2$. The long root $\sl_2$ is given in Section 4.1.1 in
\cite{PollackG2} or equivalently, by combining the discussion in Section 5.1 and Section 4.2.4 of \cite{PollackQDS}. From Section 4.1.1 in
\cite{PollackG2}, we obtain
\[
E = \frac{1}{4} \big(e_1 + e_1^* - i(e_3 + e_3^*)) \wedge (\epsilon_2-\epsilon_1 - i(e_2+e_2^*)\big),
\]

$F = - \overline{E}$ and $H = [E,F]$.

With the identification in \eqref{eqn:bE}, it follows that the long root $\sl_2$ of $G_2$ maps into the third copy of $\sl_2$ in $\textrm{Lie}(K_\infty)\otimes \C$ that was given above.

We denote by $\C^2$ the representation of $\textrm{Lie}(K_\infty) \otimes \C$ which is the two-dimensional representation of the long root $\sl_2$ and the trivial representation of the other $\sl_2$'s. Let $\{x,y\}$ be a basis of $\C^2$ on elements on which $H$ acts as $1,-1$, in order, and for which $F x = y$. The even symmetric powers in $\textrm{Sym}^{2 \ell}(\C^2)$ exponentiate to representations of $K_\infty$ and have the basis $\{x^{2\ell},x^{2 \ell-1}y, \ldots, x y^{2 \ell-1}, y^{2 \ell}\}.$

\subsection{Binary cubic forms} \label{subsec:binary cubics}
Here, we briefly recall some aspects of binary cubic forms, as studied in \cite{PollackG2}, that will be used in the next subsection.

Let $V_2$ denote the defining representation of $\GL_2$. The space
\[
W = \textrm{Sym}^3(V_2) \otimes \det(V_2)^{-1}
\]
is the space of binary cubic forms. If $f(w,z) = aw^3 + bw^2 z + cwz^2 + dz^3$ is a binary cubic and $g \in \GL_2$, then we define $g \cdot f$ to be the binary cubic
\begin{equation}\label{GL2 action}
(g\cdot f)(w,z) = \det(g)^{-1} f((w,z)g).
\end{equation}
Also, we will sometimes use a right action of $\GL_2$ on the space of binary cubics. We define
\[
\widetilde{g} =
\bigg(\begin{array}{cc} s& -q\\ -r & p\end{array}\bigg) \quad \text{for } \, g = \bigg(\begin{array}{cc}  p&q\\ r & s \end{array}\bigg)\in \GL_2,
\]
so that $g \widetilde{g} = \det(g)$. We then define
\[
f \cdot g = \widetilde{g} \cdot f = \det(g)^2 f((w,z)g^{-1}).
\]

There is a $\GL_2$-equivariant symplectic form on $W$ that is defined as
\[
\langle aw^3 + bw^2 z + cwz^2 + dz^3,a'w^3 + b'w^2 z + c'wz^2 + d'z^3\rangle
= ad' - \frac{bc'}{3} + \frac{cb'}{3} - da'.
\]
We have
\[
\langle g \cdot f, g \cdot f'\rangle = \det(g) \langle f, f' \rangle \quad \text{and} \quad \langle f, g \cdot f' \rangle = \langle f \cdot g, f' \rangle \quad \text{ for all }\, f, f' \in W.
\]

There also exists a $\GL_2$-equivariant quartic form on the space $W$. For $v = aw^3 + bw^2 z + cwz^2 + dz^3 \in W$, this is given by
\begin{align*}
q(v) &= \Big(ad - \frac{bc}{3}\Big)^2 + \frac{4}{27} ac^3 + \frac{4}{27}db^3 - \frac{4}{27} b^2 c^2 \\
&= -\frac{1}{27}(-27a^2d^2 + 18abcd+b^2c^2-4ac^3-4db^3).
\end{align*}


\subsection{Characters of the Heisenberg parabolic}\label{subsec:Heis}
Throughout the paper, we fix the standard additive character to be $\psi: \Q\backslash \A \rightarrow \C^\times$. We will abusively denote the $p$-component of this additive character by $\psi$. Thus, if $x \in \Q_p$ and $x = x_0 + x_1$ with $x_0 \in \Z_p$ and $x_1 = m/p^r$, then
\[
\psi(x) = e^{2\pi i x_1}.
\]

We let $N$ denote the unipotent radical of the Heisenberg parabolic $P$ of $G_2$ and let $M$ denote the Levi subgroup of $P$ that also stabilizes the line spanned by $E_{31}$. We identify $M$ with $\GL_2$ as in \cite[Section 5.2]{PollackG2}. We now recall this identification. Suppose that $g \in \GL_2$ is represented by the matrix  $\mm{a}{b}{c}{d}$. Then the action of $g$ on $\Theta$ is given by
\begin{itemize}
	\setlength\itemsep{0.3em}
	\item $e_1 \mapsto a e_1 + c e_3^*$,
	\item $e_3^* \mapsto b e_1 + d e_3^*$,
	\item $(ad-bc) \epsilon_1 \mapsto ad \epsilon_1 + ab e_2^* - cd e_2 -bc \epsilon_2$,
	\item $(ad-bc) e_2^* \mapsto ac \epsilon_1 + a^2 e_2^* - c^2 e_2 -ac \epsilon_2$,
	\item $(ad-bc) e_2 \mapsto -bd \epsilon_1 -b^2 e_2^* +d^2 e_2 +bd \epsilon_2$,
	\item $(ad-bc) \epsilon_2 \mapsto -bc \epsilon_1 -ab e_2^* + cd e_2 + ad \epsilon_2$.
\end{itemize}
On another note, a binary cubic form provides a character of $N$. Let $Z$ denote the one-dimensional center of $N$. Denote by $W$ the representation $\textrm{Sym}^3(V_2) \otimes \det(V_2)^{-1}$ of $M \simeq \GL_2$. The exponential map $\exp: W \to N/Z$ provides an identification $W \simeq N/Z$ as specified in \cite[p. 18]{PollackG2}. Namely, to the binary cubic
\[
u_1 x^3 + u_2 x^2 y + u_3 xy^2 + u_4 y^3 \in W,
\]
it associates the element
\[
u_{1} E_{12} + \frac{u_2}{3} v_1 + \frac{u_3}{3} \delta_3 + u_4 E_{23} \in \, \textrm{Lie}(G_2).
\]
Now, if $\omega \in W$, then $n \mapsto \psi(\langle \omega, \overline{n}\rangle)$ defines a character of $N$. Here $\overline{n}$ is the image of $n$ in $N/Z \simeq W$ and $\psi$ is our fixed additive character.


\section{The Rankin-Selberg integral} \label{RSI}
In this section we provide an overview of the calculations that will be done in the rest of the paper.


\subsection{The Eisenstein series}\label{subsec:Eis}

We begin by defining various Eisenstein series on the group $G$. Recall that $P_G$ denotes the Heisenberg parabolic of $G$. We denote its generating character by $\nu: P_G \rightarrow \GL_1$, so that
\[
\delta_{P_G}(p) = |\nu(p)|^{5}.
\]

Let $E_\ell(g,s)$ be the Eisenstein series of weight $\ell$ on $G$ that is normalized with a flat section. More precisely, if $x$ and $y$ denote the variables in $\operatorname{Sym}^{2\ell}(\C^2)$, we define $f_{\ell}(g,s)$ to be the unique section in $\mathrm{Ind}_{P_{G}(\A)}^{G(\A)}(|\nu|^{s})$, which is valued in $\mathbf{V}_{\ell} = \operatorname{Sym}^{2\ell}(\C^2)$, and satisfies the properties
\begin{enumerate}[label=(\roman*)]
\setlength\itemsep{0.3em}
	\item $f_{\ell}(k_f,s) = x^{\ell} y^{\ell}$ for all $k_f \in K_f \subseteq G(\A_f)$,
	\item $f_{\ell}(gk,s) = k^{-1} \cdot f_{\ell}(g,s)$ for all $g \in G(\A)$ and $k \in K_\infty$.
\end{enumerate}
Note that in this paper we use an unnormalized induction. We have
\[
E_{\ell}(g,s) = \sum_{\gamma \in P_{G}(\Q)\backslash G(\Q)}{f_{\ell}(\gamma g,s)}.
\]

Let $\Lambda(s)$ be the completed Riemann zeta function, that is,
\[
\Lambda(s)=\pi^{-\frac{s}{2}}\Gamma\left(\frac{s}{2}\right)\zeta(s).
\]
For our purposes, we define a normalized Eisenstein series as
\begin{equation}\label{eqn:NEis}
E^*_\ell(g,s) = \Lambda(s-1)^2 \Lambda(s) \Lambda(2s-4) \frac{\Gamma(s+\ell-1)\Gamma(s+\ell-2)}{\Gamma(s-1)\Gamma(s-2)} E_\ell(g,s).
\end{equation}

It will be convenient for our calculations to define another Eisenstein series. Let $\Phi_f$ be the Schwartz-Bruhat function on $\wedge^2 \Theta \otimes \A_f$ that is the characteristic function of $\wedge^2 \Theta_0 \otimes \widehat{\Z}$. Note that $\Phi_f$ is stable by $K_f$. For $g \in G(\A_f)$ we define
\[
f_{\operatorname{fte}}(g,\Phi_f,s) = \int_{\GL_1(\A_f)}{|t|^s \Phi_f(t g^{-1} E_{13})\,dt},
\]
where $f_{\operatorname{fte}}$ stands for the finite part of $f$. Also, let
\[
f(g,\Phi_f,s) = f_{\operatorname{fte}}(g_f,\Phi_f,s) f_{\ell}(g_\infty,s).
\]
Because $\Phi_f$ is $K_f$-stable, it is immediate that $f(g,\Phi_f,s) = \zeta(s) f_{\ell}(g,s)$. We set
\begin{equation}\label{def:E g phif s}
E(g,\Phi_f,s) = \sum_{\gamma \in P_{G}(\Q)\backslash G(\Q)}{f(\gamma g,\Phi_f,s)}.
\end{equation}
Then the Rankin-Selberg integral is defined as
\[
I(\varphi,\Phi,s) = \int_{G_2(\Q)\backslash G_2(\A)}{\{\varphi(g),E(g,\Phi_f,s)\}_K\,dg}.
\]
Recall that here $\varphi$ is a modular form of weight $\ell$ on $G_2$. In particular, $\varphi$ is valued in $\mathbf{V}_{\ell}$, and constructed from a cuspidal automorphic representation $\pi = \pi_f \otimes \pi_{\infty,\ell}$, where $\pi_{\infty,\ell}$ is a quaternionic discrete series of minimal $K$-type $\mathbf{V}_{\ell}^\vee \simeq \mathbf{V}_{\ell}$. The term $\{\varphi(g),E(g,\Phi_f,s)\}_K$ is the $K$-invariant pairing of these two $\mathbf{V}_{\ell}$-valued automorphic functions.

\subsection{The unfolded Rankin-Selberg integral}

We now explain how the Rankin-Selberg integral \\ $I(\varphi,\Phi_f,s)$ unfolds.

Let $v_{E} \in W$ denote the binary cubic
\[
v_{E}= w^2z + wz^2 = wz(w+z).
\]
Also, let $\chi$ be the character of $N(\Q)\backslash N(\A)$ determined by $v_{E}$ via the association given at the end of Section \ref{subsec:Heis} and set
\[
\varphi_\chi(g) = \int_{N(\Q)\backslash N(\A)}{\chi^{-1}(n)\varphi(ng)\,dn}.
\]

We further define
\[
\widetilde{v_E} = \epsilon_1 \wedge (e_1 + e_3^*),
\]
which is an element of $\wedge^2 \Theta$. Let $N^{0,E} \subseteq N$ be the subgroup consisting of those $n \in N$ for which $\langle v_{E}, \overline{n} \rangle = 0$.

\begin{theorem}
We have
	\begin{equation}
	I(\varphi,\Phi_f, s) = \int_{N^{0,E}(\A)\backslash G(\A)}\{f(\gamma_0g, \Phi_f, s),\varphi_{\chi}(g)\}_K \,dg,
	\end{equation}
where $\gamma_0 \in G(\Q)$ satisfies $\gamma_0^{-1} E_{13} = \widetilde{v_E}$.
\end{theorem}

The proof of this theorem is due to Gurevich and Segal~\cite{GurevichSegal, Segal2017}, but its above form is essentially Theorem 5.2 in~\cite{PollackG2}.


\subsection{Local integrals}
\label{local integrals}
In order to analyze $I(\varphi,\Phi_f,s)$, we must consider the associated local integral at each place of $\Q$. In this section, we describe these local integrals and provide an outline of their computation to be done in the later sections.

The integral at the archimedean place is given by
\[
I(s;\ell) = \int_{N^{0,E}(\R)\backslash G_2(\R)}{\{f_{\ell}(\gamma_0 g,s), \mathcal{W}_\chi(g)\}_K\,dg}.
\]
Here $\mathcal{W}_\chi$ is the generalized Whittaker function of \cite[Section 4]{PollackG2} and \cite{PollackQDS}. Now, in view of the normalization of the Eisenstein series in \eqref{eqn:NEis}, note that
\[
\Gamma_\R(s-1)^2 \Gamma_\R(s) \Gamma_\R(2s-4) \frac{\Gamma(s+\ell-1)\Gamma(s+\ell-2)}{\Gamma(s-1)\Gamma(s-2)}
= 2^{s} \Gamma_\R(s-1) \Gamma_\C(s+\ell-1)\Gamma_\C(s+\ell-2).
\]
Thus we suitably define a normalized archimedean zeta integral as
\begin{equation*}
I^*(s;\ell) =2^{s} \Gamma_\R(s-1) \Gamma_\C(s+\ell-1)\Gamma_\C(s+\ell-2)  I(s;\ell).
\end{equation*}
This integral will be computed in Section \ref{sec:archInt}. In Theorem \ref{thm:I*int}, we will prove that $I^*(s;\ell) = L(\pi_{\ell,\infty},s-2)$ up to a nonzero constant.

We now define the local integrals at the finite places. Let $\Phi_p$ denote the characteristic function of $\wedge^2 \Theta_0 \otimes \Z_p$ and
\[
f_p(g,\Phi_p,s) := \int_{\GL_1(\Q_p)}{|t|^s \Phi_p(t g^{-1} E_{13})\,dt}.
\]
Here $f_p(g,\Phi_p,s)$ is also the associated local inducing section, so that $f_p(1,\Phi_p,s) = \zeta_p(s)$. Let $V_{\pi_p}$ denote the space of the representation $\pi_p$, and write $v_0$ for a spherical vector. Suppose that $\mathcal{L}: V_{\pi_p} \rightarrow \C$ is an $(N,\chi)$-functional, that is,
\[
\mathcal{L}(n v) = \chi(n) \mathcal{L}(v) \quad \text{for all }\, n \in N(\Q_p)  \text{ and } \, v \in V_{\pi_p}.
\]
At a finite place $p$, we will compute
\begin{align}
	\label{Ips}
	I_p(s) = I_p(\mathcal{L},s)
	&:= \int_{N^{0,E}(\Q_p)\backslash G_2(\Q_p)}{f_p(\gamma_0 g,\Phi_p,s) \mathcal{L}(g v_0)\,dg} \\
	&= \int_{N(\Q_p) \backslash \GL_1(\Q_p) \times G_2(\Q_p)}{|t|^{s} \Phi_{p,\chi}(t,g) \mathcal{L}(g v_0)\,dt\,dg}, \nonumber
\end{align}
where
\begin{equation}\label{Phichi}
\Phi_{p,\chi}(t,g) = \int_{N^{0,E}(\Q_p)\backslash N(\Q_p)}{\chi(n) \Phi_p(t g^{-1}n^{-1} \widetilde{v_{E}})\,dn}.
\end{equation}
We will prove the following theorem.

\begin{theorem}\label{thm:Ip}
We have
\begin{equation}\label{eq:I p result}
I_p(s)=
I_p(\mathcal{L};s) =\mathcal{L}(v_0) \frac{L(\pi_p,\operatorname{Std},s-2)}{\zeta_p(s-1)^2 \zeta_p(2s-4)}.
\end{equation}
\end{theorem}


The proof of this theorem proceeds by following the strategies in \cite{GurevichSegal,Segal2017,PollackG2}.

Let $V_7$ be the perpendicular subspace to $1$ in $\Theta$, and set $V_7(\Z) = V_7 \cap \Theta_0$. Similarly, define $V_7(\Z_p) = V_7(\Z) \otimes \Z_p$. Note that because $G_2$ stabilizes $1$, $V_7(\Z_p)$ is stabilized by $G_2(\Z_p)$. We write $r_7: G_2 \rightarrow \GL(V_7)$ for the action map.

We now define two Hecke operators on $G_2$. First, for $t \in \GL_1(\Q_p)$ and $h \in G_2(\Q_p)$, let
\[
\Delta(t,h) = \charf\big(t \cdot r_7(h) \in \textrm{End}(V_7(\Z_p))\big).
\]
We call this the approximate basic function (see Section 5.3 in \cite{PollackG2} for some remarks on this terminology). Define another Hecke operator on $G_2$ as
\[
\mathcal{T}=p^{-3}\times\charf\big(g \in G_2(\Q_p),~~ p.r_7(g) \in \textrm{End}(V_7(\Z_p))\big).
\]
For ease of notation, let $z=p^{-s}$.
In order to prove \eqref{eq:I p result}, we will prove that
\begin{equation}\label{mainequation}
\begin{split}
	\int_{\GL_1(\Q_p)\times G_2(\Q_p)}
	&|t|^{s+2}\Delta(t,g)\mathcal{L}(gv_0)\,dt\,dg \\
	&= M(\pi_p, s)\int_{N(\Q_p)\backslash GL_1(\Q_p)\times G_2(\Q_p)}|t|^{s+1}\Phi_{p,\chi}(t,g)\mathcal{L}(gv_0)\,dt\,dg,
\end{split}
\end{equation}
where
\[
M(\pi_p, s) = (1-pz)(1-z)N_0(\pi_p, s-1)\zeta_p(s)^2\zeta_p(2s-2),
\]
and
\[
N_0(\pi_p,s)v_0 = 1+ (p^{-1}+1)z + \frac{z^2}{p} + (p^{-2}+p^{-1})z^3 + \frac{z^4}{p^2}- \frac{z^2}{p}T.
\]
Proving \eqref{mainequation} implies our desired relation between $I_p(\mathcal L;s)$ and $L(\pi_p,\operatorname{Std},s)$ as in Theorem~\ref{thm:Ip}, which is essentially Proposition 7.1 in \cite{GurevichSegal}. For an explanation of this implication, see Section 5.3 in \cite{PollackG2}.

Now, \eqref{mainequation} has been proved for $p \geq 5$ in \cite{GurevichSegal} and \cite{PollackG2}. We will prove it for $p =2,3$ as well. To do so, we will compute the left-hand side and the right-hand side of \eqref{mainequation} separately and show that they are the same. The left-hand side will be considered in Section \ref{sec:ABF} and the right-hand side will be computed in Section \ref{sec:non-arch}. In order to do these computations, some properties of binary cubic forms and their relations to cubic rings will be useful. We spell these out in the next section.


\section{The arithmetic invariant theory of binary cubics}\label{Sect. binary cubic}

In this section, we describe some results on the relationship between binary cubic forms and cubic rings. We refer to \cite[Section 4]{GanGrossSavin} and \cite{GrossLucianovic} for a primer on this relationship.

Suppose that
\[
f(w,z) = a w^3 + b w^2 z + c wz^2 + d z^3
\]
is a binary cubic over some ring $R$. One associates to $f$ the cubic $R$-algebra $T$ with the basis $\{1,\omega,\theta\}$ and the multiplication table
\begin{itemize}
\setlength\itemsep{0.2em}
	\item $\omega \theta = - a d$,
	\item $\omega^2 = -a c + a \theta - b \omega$,
	\item $\theta^2 = - b d + c \theta -d \omega$.
\end{itemize}
Following \cite[Section 4]{GanGrossSavin}, we call such a basis a good basis. Suppose that $m = \mm{m_{11}}{m_{12}}{m_{21}}{m_{22}}$ is a $2\times 2$ matrix with coefficients in $R$. Write $T_m$ for the $R$-lattice in $T$ that is spanned by $\{1, m_{11}\omega + m_{12}\theta, m_{21}\omega + m_{22}\theta\}$. Let $\omega''$ and $\theta''$ be defined by the relation
\begin{equation}\label{defn w'' theta''}
\Big(\begin{array}{c} \omega''\\ \theta'' \end{array}\Big) = m \Big(\begin{array}{c} \omega \\ \theta \end{array}\Big)= \mb{m_{11}}{m_{12}}{m_{21}}{m_{22}}\Big(\begin{array}{c} \omega \\ \theta \end{array}\Big).
\end{equation}
One can naturally ask what condition guarantees that $T_m$ is closed under multiplication. This question is answered by the following proposition.

\begin{proposition} \label{AIT}
Suppose that $R$ has characteristic $0$. Set
\[
f'(w,z) = m \cdot f(w,z),
\]
where the action is as given in \eqref{GL2 action}, and write
\[
f'(w,z) = a' w^3 + b' w^2 z + c' wz^2 + d' z^3.
\]
With the notation as above, the $R$-lattice $T_m$ is closed under multiplication if and only if
	\begin{enumerate}[label=(\roman*)]
	\setlength\itemsep{0.3em}
		\item $f'(w,z)$ has coefficients in $R$, that is, $a',b',c',d' \in R$,
		\item and $\bigg(\begin{array}{c} b' \\-c'\end{array}\bigg) \equiv m \bigg(\begin{array}{c} b \\ -c \end{array}\bigg) \pmod{3}$.
	\end{enumerate}
\end{proposition}

The proof of this proposition requires a lemma. Following the notation in \eqref{defn w'' theta''}, set $\omega'' =  m_{11}\omega + m_{12}\theta$ and $\theta'' = m_{21}\omega + m_{22}\theta$. Let $\omega'$ and $\theta'$ in $T\otimes \textrm{Frac}(R)$ be defined by
\[
\bigg(\begin{array}{c} \omega' \\ \theta' \end{array}\bigg)
= m \bigg(\begin{array}{c} \omega \\ \theta \end{array}\bigg) + \frac{1}{3} \bigg\{\bigg(\begin{array}{c} b' \\ -c' \end{array}\Big) -m\Big(\begin{array}{c} b \\ -c \end{array}\bigg) \bigg\}.
\]
Finally, let
\[
\omega_0 = \omega +\frac{b}{3} \quad \text{and} \quad \theta_0 = \theta-\frac{c}{3},
\]
so that $\tr(\omega_0) = \tr(\theta_0) = 0$.

\begin{lemma}\label{lem:multTable}
The elements $\omega'$ and $\theta'$ defined above have the multiplication table
	\begin{enumerate}[label=(\roman*)]
	\setlength\itemsep{0.3em}
		\item $\omega' \theta' = - a' d'$,
		\item $\omega'^2 = -a' c' + a' \theta' - b' \omega'$,
		\item $\theta'^2 = - b' d' + c' \theta' -d' \omega'$.
	\end{enumerate}
\end{lemma}
\begin{proof}
This lemma is well-known. It is essentially the statement that our association of cubic rings to binary cubic forms is equivariant under the action of $\GL_2$. For completeness, we give some details regarding a proof.
	
First, instead of checking the multiplication table above of $\{1, \omega',\theta'\}$, we verify the equivalent multiplication table for $\Big(\begin{array}{c} \omega'_0 \\ \theta'_0 \end{array}\Big) = m \Big(\begin{array}{c} \omega_0 \\ \theta_0 \end{array}\Big)$.

Now, the trace $0$ basis $\{\omega_0$, $\theta_0\}$ has the multiplication table
	\begin{itemize}
	\setlength\itemsep{0.6em}
		\item $\displaystyle \omega_0 \theta_0 = \frac{b}{3} \theta_0 - \frac{c}{3} \omega_0 + \left(\frac{bc}{9}-ad\right)$,
		\item $\displaystyle  \omega_0^2 = a \theta_0 - \frac{b}{3} \omega_0 + \frac{2}{9}(b^2-3ac)$,
		\item $\displaystyle  \theta_0^2 = \frac{c}{3} \theta_0 - d\omega_0 + \frac{2}{9}(c^2 - 3bd)$.
	\end{itemize}
We wish to prove that the multiplication table for $\{\omega_0',\theta_0'\}$ has the same form, with $a,b,c,d$ replaced by $a',b',c',d'$, respectively. To do this, we first write the multiplication table of $\{\omega_0,\theta_0\}$ as
\[
\bigg(\begin{array}{c} \omega_0 \\ \theta_0 \end{array}\bigg) \big(\begin{array}{cc} \omega_0 & \theta_0 \end{array}\big) = \frac{1}{3} \mb{3a}{b}{b}{c} \theta_0 - \frac{1}{3}\mb{b}{c}{c}{3d} \omega_0 + \frac{1}{9} \mb{2b^2-6ac}{bc-9ad}{bc-9ad}{2c^2-6bd}.
\]
Then we obtain
\begin{align*}
m\bigg(\begin{array}{c} \omega_0 \\ \theta_0 \end{array}\bigg) \big(\begin{array}{cc} \omega_0 & \theta_0 \end{array}\big)m^t
=\, & \frac{1}{3}\left(m^{-1}_{21} m \mb{3a}{b}{b}{c} m^t - m^{-1}_{11} m \mb{b}{c}{c}{3d} m^t\right) \omega_0' \\
&\,+ \frac{1}{3}\left(m^{-1}_{22} m \mb{3a}{b}{b}{c} m^t - m^{-1}_{12} m \mb{b}{c}{c}{3d} m^t\right) \theta_0' \\
&\,+ \frac{1}{9}m \mb{2b^2-6ac}{bc-9ad}{bc-9ad}{2c^2-6bd}m^t.
\end{align*}
Equivalently, this is
\begin{align*}
\bigg(\begin{array}{c} \omega_0' \\ \theta_0' \end{array}\bigg) \big(\begin{array}{cc} \omega_0' & \theta_0' \end{array}\big)
= &\, -\frac{1}{3}\det(m)^{-1} \left(m_{21} m \mb{3a}{b}{b}{c} m^t + m_{22} m \mb{b}{c}{c}{3d} m^t\right) \omega_0' \\
&\,+ \frac{1}{3}\det(m)^{-1}\left(m_{11} m \mb{3a}{b}{b}{c} m^t + m_{12} m \mb{b}{c}{c}{3d} m^t\right) \theta_0' \\
&\, + \frac{1}{9}m \mb{2b^2-6ac}{bc-9ad}{bc-9ad}{2c^2-6bd}m^t.
\end{align*}
Then
\[
\bigg(\begin{array}{c} \omega_0' \\ \theta_0' \end{array}\bigg) \big(\begin{array}{cc} \omega_0' & \theta_0' \end{array}\big) = \frac{1}{3} \mb{3a'}{b'}{b'}{c'} \theta_0 - \frac{1}{3}\mb{b'}{c'}{c'}{3d'} \omega_0 + \frac{1}{9} \mb{2b'^2-6a'c'}{b'c'-9a'd'}{b'c'-9a'd'}{2c'^2-6b'd'},
\]
where we used the definition of the action of $\GL_2$ on binary cubics and the equivariance of the Hessian of a binary cubic. The lemma then follows.
\end{proof}

We denote the condition (ii) of Proposition \ref{AIT} by $\dagger$ as below.
\[
\bigg(\begin{array}{c} b' \\-c'\end{array}\bigg) \equiv m \bigg(\begin{array}{c} b \\ -c \end{array}\bigg) \pmod{3}. \tag*{$\dagger$}
\]
Then the statement of Proposition \ref{AIT} follows immediately from

\begin{proposition} \label{Prop cubic}
	The following statements are equivalent.
	\begin{enumerate}[label=(\roman*)]
		\item The $R$-lattice $T_m$ spanned by $\{1,\omega'',\theta''\}$ is closed under multiplication.
		\item The $R$-lattice spanned by $\{1,\omega',\theta'\}$ is closed under multiplication and $\dagger$ holds.
		\item $m \cdot f$ has coefficients in $R$ and $\dagger$ holds.
	\end{enumerate}
\end{proposition}

\begin{proof}
From Lemma \ref{lem:multTable}, it is clear that (ii) and (iii) are equivalent. It is also clear that (ii) implies (i). To prove that (i) implies (ii), we argue as follows. First, we define
\[
\bigg(\begin{array}{c} \delta_1 \\ \delta_2 \end{array}\bigg) = \frac{m}{3} \bigg(\begin{array}{c} b \\ -c \end{array}\bigg) + \frac{1}{3}\bigg(\begin{array}{c} -b' \\ c' \end{array}\bigg),
\]
so that $\omega' = \omega'' + \delta_1$ and $\theta' = \theta'' + \delta_2$. Observe that
\[
\omega'' \theta'' =  (\omega'-\delta_1)(\theta'-\delta_2)
= B -\delta_1 \theta'' -\delta_2 \omega''
\]
for some $B \in \textrm{Frac}(R)$. Thus if (i) holds, then $\delta_1, \delta_2 \in R$. Since $\delta_1, \delta_2 \in R$, the equalities $\omega'  = \omega'' + \delta_1$, $\theta' = \theta''+ \delta_2$ imply that the $R$-lattice spanned by $\{1,\omega',\theta'\}$ is closed under multiplication, so that (ii) holds. The result then follows.
\end{proof}

In the case where $R = \Z_p$, we can go further.

\begin{proposition}\label{prop:daggerHolds}
If $ \mathrm{Span}_{\Z_p}(1, \omega',\theta')$ is closed under multiplication, then $\dagger$ holds. Equivalently, if $m \cdot f$ has its coefficients in $\Z_p$, then $\dagger$ holds.
\end{proposition}

Although this proposition has nontrivial content only when $p=3$, we write down its proof for general $p$.

\begin{proof}
The idea is to use the Cartan decomposition of $m \in \GL_2(\Q_p) \cap M_2(\Z_p)$. That is, such an $m$ is a product $k_1 t k_2$ for some $k_1, k_2 \in \GL_2(\Z_p)$ and a diagonal element $t$ in $\GL_2(\Q_p) \cap M_2(\Z_p)$.

We first claim that if $m=k \in \GL_2(\Z_p)$, then $\dagger$ automatically holds for this $m$. To see this, note that because $k \in \GL_2(\Z_p)$, $\textrm{Span}_{\Z_p}(1,\omega, \theta) = \textrm{Span}_{\Z_p}(1,\omega'', \theta'')$. Thus that $\dagger$ holds follows from the equivalence of (i) and (iii) of Proposition \ref{Prop cubic}.

Now suppose that $m = t = \diag(t_1,t_2)$ is diagonal in $M_2(\Z_p) \cap \GL_2(\Z_p).$ Then for such $m$, $b' = t_1 b$ and $c' = t_2 c$, and it is clear that if $t \cdot f$ has coefficients in $\Z_p$, then $\dagger$ holds.

Now suppose that $m = k_1 t k_2$ and that $m \cdot f$ has coefficients in $\Z_p$. It follows that $t k_2 \cdot f$ has coefficients in $\Z_p$. Thus, from what has been said, $\dagger$ holds for $m'=t k_2$. Applying $k_1$, it follows that $\dagger$ holds for $m$, as desired.
\end{proof}

 Recall that an order is a subring of a $K$-algebra, where $K$ is a field and $R$ is an integral domain in $K$, which is a full $R$-lattice. We note the following corollary of Lemma \ref{lem:multTable}.

\begin{corollary}\label{cor:mfint}
Set $R = \Z_p$, and let the binary cubic $f(w,z)$ correspond to the maximal order $T$ in the \'etale $\Q_p$-algebra $T \otimes \Q_p$. Assume that for $m \in \GL_2(\Q_p)$, the coefficients of $m \cdot f$ lie in $\Z_p$. Then $m \in M_2(\Z_p)$.
\end{corollary}

\begin{proof}
Suppose that $\displaystyle m = \mm{m_{11}}{m_{12}}{m_{21}}{m_{22}}$ and $T$ has the good basis $\{1, \omega, \theta \}$. We have $\omega' = m_{11}\omega + m_{12} \theta + \delta_1$ and $\theta' = m_{21} \omega+ m_{22}\theta + \delta_2$. Since $m \cdot f$ has coefficients in $\Z_p$, Lemma \ref{lem:multTable} implies that $\mathrm{Span}_{\Z_p}\big(1,\omega',\theta'\big)$ is closed under multiplication. But then, because $T$ is maximal by assumption, we must have $\omega', \theta' \in T$. It follows that all entries of $m$ are in $\Z_p$, as desired.
\end{proof}


\section{The Fourier coefficient of the approximate basic function}\label{sec:ABF}
In this section, we explain the computation of the left-hand side of \eqref{mainequation}. For $t \in \GL_1$ and $h \in \GL_2 \simeq M$ the Levi of the Heisenberg parabolic of $G_2$, define
\[
D_{\chi}(t,h) =  \int_{N(\Q_p)}\chi(n)\Delta(t,nh)\,dn.
\]
Then by the Iwasawa decomposition, the left-hand side of \eqref{mainequation} is
\[
D(s) = \int_{\GL_1(\Q_p)\times \GL_2(\Q_p)}\delta_P^{-1}(h)|t|^{s+2}D_{\chi}(t,h)L(h v_0)\,dh\,dt.
\]
For $p \geq 5$, the $D_\chi(t,h)$ is computed in \cite[Proposition 5.7]{PollackG2}. Below we explain that the expression obtained for $D_\chi(t,h)$ in \emph{loc cit} continues to hold for $p=2$ and $p=3$.

We now recall various notations from \cite{PollackG2} that we need to state for the computation of $D_\chi(t,h)$. First, let
\[
f_{\textrm{max}}(w, z) = aw^3 + bw^2z +cwz^2 + dz^3
\]
be a binary cubic form corresponding to the maximal order
\[
\mathcal{O}_E =\Z_p \times \Z_p \times \Z_p \in \Q_p^3,
\]
so that $f_{\textrm{max}}$ is some $\GL_2(\Z_p)$ translate of $wz(w+z)$. Let $\{1, \omega,\theta\}$ be the good basis of $\mathcal{O}_E$ associated to $f_{\textrm{max}}$. For $x = \mm{\alpha}{\beta}{\gamma}{\delta} \in \GL_2(\Q_p)$, $T(x)$ denotes the $\Z_p$-module spanned by $\{1, \delta \omega - \beta \theta,-\gamma \omega + \alpha \theta\}$. Hence $T(x) = T_{\widetilde{x}}$ in the notation of Section \ref{Sect. binary cubic}, where $\widetilde{x}$ is as defined in Section~\ref{subsec:binary cubics}. Note that by the results of Section \ref{Sect. binary cubic}, if $T(x)$ is closed under multiplication, then $x \in M_2(\Z_p)$ and $\widetilde{x} \cdot f_{\textrm{max}} = f_{\textrm{max}} \cdot x$ has its coefficients in $\Z_p$, and vice versa.

For a general binary cubic form $\Omega$ with coefficients in $\Z_p$, define $\mathcal{N}(\Omega)$ to be the number of $0$'s of $\Omega$ in $\textbf{P}^1(\F_p)$. Also, for an element $h \in \GL_2(\Q_p)$, define $\val(h)\in \Z$ to be the largest integer $n$ so that $p^{-n}h \in M_2(\Z_p)$.

\begin{proposition}
Define $x_0(h)$ by $x_0(h) = p^{-\val(h)}h$, and set $\lambda = \det(h)/t.$ Write $D_{\chi}'(\lambda, h) = D_{\chi}(t,h)$, that is, $D_{\chi}'$ is the same function as $D_{\chi}$, except that it is expressed in terms of the new variables $\lambda$ and $h$. Further, let
\[
\epsilon(x_0(h)) = \left\{
    \begin{array}{ll}
         1 & \mbox{if} ~ ~ x_0(h) \in \GL_2(\Z_p), \\
         2 & \mbox{if} ~ ~ x_0(h) \notin \GL_2(\Z_p).
    \end{array}
\right.
\]
Then
\begin{align*}
D_{\chi}'(\lambda, h) &= |\det(\lambda^{-1}h)|^{-1}\charf\big(h\in M_2(\Z_p), \val(\lambda^{-1}h) \in \{0,1\},  T(x_0(h)) ~ \text{a ring}\big)
\\ &\,\,\, \times  \left\{
    \begin{array}{ll}
         1 & \mbox{if} ~ ~ \val(\lambda^{-1}h)=0 \\
         N(f_{\operatorname{max}}) - \epsilon(x_0(h)) & \mbox{if} ~ ~ \val(\lambda^{-1}h)=1
    \end{array}
\right\}.
\end{align*}
\end{proposition}

\begin{proof}
This is essentially Proposition 5.7 in \cite{PollackG2}. The proof carries over line-by-line except one minor change. To aid the reader in checking this, we give some of the omitted details from \emph{loc cit} to clarify that the result continues to hold for $p=2,3$.
	
First, we show that $D_\chi(t,h) \neq 0$ implies $h \in M_2(\Z_p)$. We have $D_{\chi}(t,h) = \int_{N}{\psi(\langle \omega, n \rangle) \Delta(t, nh)\,dn}$, where $\omega$ is the element of $W$ that corresponds to $f_{\textrm{max}}$. By the change of variable $n \mapsto hnh^{-1}$, we find that up to positive constant coming from the change in measure,
\[
D_{\chi}(t,h) = \int_{N}{\psi(\langle \omega, hnh^{-1} \rangle)\Delta(t,hn)dn}.
\]
Now, $\Delta$ is right-invariant under $G_2(\Z_p)$, so if $u_0 \in G_2(\Z_p) \cap N$, then $\Delta(t,hn u_0) = \Delta(t,hn)$. Then by the change of variable $n \mapsto n u_0$ in $D_\chi(t,h)$, one finds that
\[
D_{\chi}(t,h) = \psi(\langle \omega, hu_0h^{-1}\rangle) D_{\chi}(t,h).
\]
Thus, for $D_\chi(t,h)$ to be nonzero, one must have $\langle \omega, hu_0h^{-1} \rangle \in \Z_p$  for every $u_0 \in N \cap G_2(\Z_p)$. It follows that $\omega \cdot h$ corresponds to a binary cubic form with $\Z_p$ integral coefficients, and thus $h \in M_2(\Z_p)$ by Corollary \ref{cor:mfint}.

Let us remark upon the one aspect of the proof which is ever so slightly different from the proof of Proposition 5.7 in \cite{PollackG2}. In \emph{loc cit}, one verifies that if $D_\chi(t,h)$ nonzero, then $f_{\textrm{max}} \cdot x_0(h)$ has coefficients in $\Z_p$. By Proposition \ref{prop:daggerHolds}, one concludes that $T(x_0(h))$ is a ring.

The rest of the proof is as that of Proposition 5.7 in \cite{PollackG2}.
\end{proof}


\section{Non-archimedean zeta integral} \label{sec:non-arch}

In this section we compute the right-hand side of \eqref{mainequation}. In the case when $p \geq 5$, the calculation is done in \cite{PollackG2}, so the new work is for $p=2$ and $3$. Still, many computations are similar to the ones in the proof for the case $p\geq 5$.


\subsection{The computation of \texorpdfstring{$\Phi_{p,\chi}$}{}}\label{subsect:compute Phip}
To compute the right-hand side of \eqref{mainequation}, we first compute the function $\Phi_{p, \chi}(t,g)$ in \eqref{Phichi}. The computation of this function is different from the one in Lemma 5.6 of \cite{PollackG2}.

\begin{lemma}\label{Lemma for Phichi}
	Suppose that $\displaystyle h = \mm{a}{b}{c}{d}$ is in the Heisenberg Levi, so that $h$ takes $e_1$ to $a e_1 + c e_3^*$ and $e_3^*$ to $b e_1 + d e_3^*$. Let $f_0(w,z) = w^2z+wz^2$. Set $\lambda = \det(h)/t$ and $h' = \frac{1}{\lambda} \widetilde{h} = \lambda^{-1} \mm{d}{-b}{-c}{a}$. Set $f_1(w,z) = h' \cdot f_0(w,z)$, and write $f_i(w,z) = \alpha_i w^3 + \beta_i w^2 z + \gamma_i wz^2 + \delta_i z^3$ for $i = 0,1$. Then
	\[
	\Phi_{p,\chi}(t,h)=|\lambda|A_0(\lambda, h),
	\]
	where $A_0(\lambda, h)$ is the characteristic function of the quantities
	\begin{itemize}
	\setlength\itemsep{0.25em}
		\item $\lambda \in \mathbb Z $,
		\item $h' \in \mathbb Z$,
		\item $h' \cdot f_0(w,z) = \det(h')^{-1} f_0((w,z) h') \in \mathbb Z $,
		\item $\bigg(\begin{array}{c} \beta_1 \\-\gamma_1\end{array}\bigg) \equiv h' \bigg(\begin{array}{c} \beta_0 \\ -\gamma_0 \end{array}\bigg) \pmod{3}$.
	\end{itemize}
\end{lemma}

Before we prove this lemma, we introduce some notation and state a corollary of it. For a cubic ring $T$ over $\Z_p$, the largest integer $c$ so that $T=\Z_p+p^cT_0$ for a cubic ring $T_0$ over $\Z_p$ is called the \emph{$p$-adic content} of $T$ and is denoted by $c(T)$. If $T$ corresponds to the binary cubic $g$, then the $p$-adic content of $T$ is the largest integer $c$ so that $p^{-c}g$ has coefficients in $\Z_p$.  Let $x=\mm{\alpha}{\beta}{\gamma}{\delta}\in \GL(2,\Q_p)$. Recall that $T(x)$ denotes the $\Z_p$-module spanned by $\{1, \delta\omega-\beta\theta, -\gamma\omega + \alpha\theta\}$, where $\{1, \omega, \theta\}$ is the good basis of the fixed maximal order.

\begin{corollary}\label{cor:A0}
We have
\[
A_0(\lambda,h)={\rm char}\big(\lambda \in \Z_p, T(\lambda^{-1}h) \text{ is a ring}\big)={\rm char}\left(\lambda \in \Z_p, \lambda\mid p^{c\left(T(h)\right)}\right).
\]
\end{corollary}

\begin{proof}
This follows from Lemma \ref{Lemma for Phichi} by an application of the results of Section~\ref{Sect. binary cubic}. The condition $\lambda \in \Z_p$ corresponds to the first bullet point while the condition $\lambda\mid p^{c\left(T(h)\right)}$ corresponds to the last three bullet points in Lemma~\ref{Lemma for Phichi}.
\end{proof}

Note that the characteristic function in the statement of Corollary \ref{cor:A0} is the same characteristic function as described in \cite[p. 21]{PollackG2} for $p\ge5$.

We now introduce some notation that will be used in the proof of Lemma \ref{Lemma for Phichi}. Recall that $\Theta_0$ is the split model of the integral octonions with $\Z$-basis $\{\epsilon_1,\epsilon_2, e_1, e_2, e_3, e_1^*, e_2^*,e_3^*\}$. For $\alpha_j$ and $\gamma_k$ in a ring $R$, we define
\[
\{\alpha_1,\alpha_2,\alpha_3\}_{e_1} := \alpha_1 \epsilon_1 \wedge e_1 -\alpha_2 \epsilon_2 \wedge e_1 + \alpha_3 e_2^* \wedge e_3^*,
\]
\[
\{\alpha_1,\alpha_2,\alpha_3\}_{e_2} := \alpha_1 \epsilon_1 \wedge e_2 -\alpha_2 \epsilon_2 \wedge e_2 + \alpha_3 e_3^* \wedge e_1^*,
\]
\[
\{\alpha_1,\alpha_2,\alpha_3\}_{e_3} := \alpha_1 \epsilon_1 \wedge e_3 -\alpha_2 \epsilon_2 \wedge e_3 + \alpha_3 e_1^* \wedge e_2^*,
\]
and
\[
\{\gamma_1,\gamma_2,\gamma_3\}_{e_1^*} := \gamma_2 \epsilon_1 \wedge e_1^* -\gamma_1 \epsilon_2 \wedge e_1^* + \gamma_3 e_2 \wedge e_3,
\]
\[
\{\gamma_1,\gamma_2,\gamma_3\}_{e_2^*} := \gamma_2 \epsilon_1 \wedge e_2^* -\gamma_1 \epsilon_2 \wedge e_2^* + \gamma_3 e_3 \wedge e_1,
\]
\[
\{\gamma_1,\gamma_2,\gamma_3\}_{e_3^*} := \gamma_2 \epsilon_1 \wedge e_3^* -\gamma_1 \epsilon_2 \wedge e_3^* + \gamma_3 e_1 \wedge e_2
\]
as elements of $\wedge^2_{\Z}\Theta_0 \otimes R$. This notation is useful because one has
\[
\big[\{\alpha_1,\alpha_2,\alpha_3\}_{e_1},\{\beta_1,\beta_2,\beta_3\}_{e_2}\big] = \{\alpha_2\beta_3+\alpha_3\beta_2,\alpha_3\beta_1 + \alpha_1 \beta_3,\alpha_1\beta_2 + \alpha_2\beta_1\}_{e_3^*}
\]
and
\[
\big[\{\gamma_1,\gamma_2,\gamma_3\}_{e_1^*},\{\delta_1,\delta_2,\delta_3\}_{e_2^*}\big] = \{\gamma_2\delta_3+\gamma_3\delta_2,\gamma_3\delta_1 + \gamma_1 \delta_3,\gamma_1\delta_2 + \gamma_2\delta_1\}_{e_3}.
\]
We also obtain
\begin{equation}
\begin{split}
\label{bracket e1e3*}
  \big[\{\alpha_1,\alpha_2,\alpha_3\}_{e_1},\{\gamma_1,\gamma_2,\gamma_3\}_{e_3^*}\big] &= (\alpha_1 \gamma_1 + \alpha_2\gamma_2 + \alpha_3 \gamma_3) e_1 \wedge e_3^*,\\
  \big[\{\alpha_1,\alpha_2,\alpha_3\}_{e_j},\{\beta_1,\beta_2,\beta_3\}_{e_j}\big] &= 0.
  \end{split}
\end{equation}
For our later use, we record the following formulas. From the formulas in Section \ref{subsec:Heis}, we find that under the action of $m = \mm{a}{b}{c}{d} \in \GL_2 \simeq M$, we have
\[
2 \epsilon_1 \wedge e_1 + \epsilon_2 \wedge e_1 - e_2^* \wedge e_3^* \mapsto a\left(2 \epsilon_1 \wedge e_1 + \epsilon_2 \wedge e_1 - e_2^* \wedge e_3^*\right) + c\left(2 \epsilon_2 \wedge e_3^* + \epsilon_1 \wedge e_3^* + e_1 \wedge e_2\right),
\]
and
\[
2 \epsilon_1 \wedge e_3^* + \epsilon_2 \wedge e_3^* - e_1 \wedge e_2 \mapsto b\left(\epsilon_1 \wedge e_1 + 2 \epsilon_2 \wedge e_1 + e_2^* \wedge e_3^*\right) + d\left(2 \epsilon_1 \wedge e_3^* + \epsilon_2 \wedge e_3^* - e_1 \wedge e_2\right).
\]
In other words,
\[
\{2,-1,-1\}_{e_1} \mapsto a \{2,-1,-1\}_{e_1} + c\{-2,1,1\}_{e_3^*},
\]
and
\[
\{-1,2,-1\}_{e_3^*} \mapsto b \{1,-2,1\}_{e_1} + d\{-1,2,-1\}_{e_3^*}.
\]

Now, recall that
\[
\widetilde{v_E} = \epsilon_1 \wedge (e_1 + e_3^*) = \epsilon_1 \wedge e_1 + \epsilon_1 \wedge e_3^* = \{1,0,0\}_{e_1} + \{0,1,0\}_{e_3^*}.
\]
We thus have
\begin{equation}\label{vE tilde}
\begin{split}
\widetilde{v_E}
&= \{1,0,0\}_{e_1} + \{0,1,0\}_{e_3^*} \\
&= \frac{1}{3}\left( \{1,1,1\}_{e_1} + \{1,1,1\}_{e_3^*}\right)
+ \frac{1}{3} \left( \{2,-1,-1\}_{e_1} + \{-1,2,-1\}_{e_3^*}\right).
\end{split}
\end{equation}
The terms in the first set of parentheses are in the Lie algebra $\mathfrak{g}_2$ and they correspond to the cubic $x^2y + xy^2$ in the sense that was described in the first paragraph of \cite[p. 18]{PollackG2}.

For $m \in M \simeq \GL_2$, we can now compute $m \widetilde{v_E}$. Write $m = \mm{a}{b}{c}{d}$. Also, let $f_0(w,z) = w^2z + wz^2$, and
\[
m \cdot f_0(w,z)
= \det(m)^{-1} f_0((w,z)m) = \alpha w^3 + \beta w^2 z + \gamma wz^2 + \delta z^3.
\]
We also recall the following notations from Section 4.1 in \cite{PollackG2}.
\begin{align*}
&E_{12} = -e_1 \wedge e_2^*,
\quad
	E_{23} = - e_2 \wedge e_3^*,
\\
	&v_1 = \{1,1,1\}_{e_1},
	\quad
 \delta_3 = \{1,1,1\}_{e_3^*}.
 \end{align*}
By using this notation and \eqref{vE tilde}, we obtain
\begin{align*}
m \widetilde{v_E}
=&\, \alpha E_{12}
+ \frac{1}{3} \beta v_1
+ \frac{1}{3} \gamma \delta_3
+ \delta E_{23} \\
&\,\,\, + \frac{1}{3}\left(a \{2,-1,-1\}_{e_1} + c\{-2,1,1\}_{e_3^*}
+ b\{1,-2,1\}_{e_1} + d\{-1,2,-1\}_{e_3^*}\right).
\end{align*}

Finally, we need to present one piece of calculation before the proof of Lemma \ref{Lemma for Phichi}. Suppose that $X = u_1 E_{12} + u_2 v_1 + u_3 \delta_3 + u_4 E_{23}$ is in the Lie algebra of $N$, that is, the unipotent radical of the Heisenberg parabolic of $G_2$. We need to compute $[X,\widetilde{v_E}]$. By \eqref{bracket e1e3*},
\begin{align*}
[X,\widetilde{v_E}] = [u_2 v_1 + u_3 \delta_3, \widetilde{v_E}]
 &= [u_2 \{1,1,1\}_{e_1} + u_3 \{1,1,1\}_{e_3^*}, \{1,0,0\}_{e_1} + \{0,1,0\}_{e_3^*}] \\
&= u_2 [\{1,1,1\}_{e_1},\{0,1,0\}_{e_3^*}]  -u_3 [\{1,0,0\}_{e_1}, \{1,1,1\}_{e_3^*}] \\
&=  (u_2 - u_3) e_1 \wedge e_3^*.
\end{align*}

We are now in a position to compute $\Phi_{p, \chi}(t,h)$.
\begin{proof}[Proof of Lemma \ref{Lemma for Phichi}]
Set $\lambda = \det(h)/t$ and  $n = \exp(X)$. From the computations above, we obtain
	\[
	n^{-1} \widetilde{v_E} = \widetilde{v_E} + (u_2-u_3) E_{13}.
	\]
Then
	\[
	t h^{-1} n^{-1} \widetilde{v_E}
	= th^{-1} \widetilde{v_E} + t\det(h)^{-1}(u_2-u_3) E_{13}
	= \lambda^{-1} \det(h) (h^{-1} \cdot \widetilde{v_E}) + \frac{u_2-u_3}{\lambda} E_{13}.
	\]
Now, $h^{-1} = \det(h)^{-1} \mm{d}{-b}{-c}{a}$. Thus        \begin{align*}
	t h^{-1} n^{-1} \widetilde{v_E}
	=&\,  \alpha_1 E_{12} + \frac{\beta_1}{3} v_1 + \frac{\gamma_1}{3} \delta_3 + \delta_1 E_{23} + \frac{u_2-u_3}{\lambda} E_{13} \\
	&+  \frac{\lambda^{-1}}{3}\left(d\{2,-1,-1\}_{e_1} -c\{-2,1,1\}_{e_3^*} -b\{1,-2,1\}_{e_1} + a\{-1,2,-1\}_{e_3^*}\right).
	\end{align*}
We will use this expression to verify the result. Indeed, by rewriting it, we obtain
\begin{align*}
    th^{-1}n^{-1} \widetilde{v_E} &= \alpha_1 E_{12} + \delta_1 E_{23} + \frac{u_2-u_3}{\lambda} E_{13}\\
    &\,\,+ \lambda^{-1} d\{1,0,0\}_{e_1} + \lambda^{-1}b \{0,1,0\}_{e_1} + \frac{1}{3}(\beta_1 - \lambda^{-1}d - \lambda^{-1}b)\{1,1,1\}_{e_1}\\
    &\,\,+ \lambda^{-1}c \{1,0,0\}_{e_3^*} + \lambda^{-1} a \{0,1,0\}_{e_3^*} + \frac{1}{3}(\gamma_1 - \lambda^{-1}c - \lambda^{-1}a)\{1,1,1\}_{e_3^*}.
\end{align*}
Observe that in order for the integral over $N^{0,E}\backslash N$ in \eqref{Phichi} to be nonzero, we must have $\lambda \in \Z_p$.  Moreover, the resulting integral is $|\lambda|$ times the characteristic function of the quantities
\begin{itemize}
    \setlength\itemsep{0.3em}
    \item $\displaystyle \alpha_1, \delta_1 \in \Z_p$,
    \item $\displaystyle \lambda^{-1}a, \lambda^{-1}b, \lambda^{-1}c, \lambda^{-1} d \in \Z_p$,
    \item $\displaystyle \frac{1}{3}(\beta_1 - \lambda^{-1}d - \lambda^{-1}b) , \frac{1}{3}(\gamma_1 - \lambda^{-1}c - \lambda^{-1}a)\in \Z_p$.
\end{itemize}
The result follows.
\end{proof}


\subsection{The local unramified computations}
Recall that $T(x)$ is the $\Z_p$-module spanned by $\{1, \delta\omega-\beta\theta, -\gamma\omega + \alpha\theta\}$ as in Section~\ref{subsect:compute Phip}. For ease of notation, we set
\[
c(x)= c(T(x)) \quad \text{for } \, x\in \GL_2(\Q_p).
\]
For an element $h \in \GL(2,\Q_p)$, let $[h]$ denote the coset $h\GL(2,\Z_p).$ Whether or not $T(x)$ is closed under multiplication is independent of the element $x \in h\GL(2,\Z_p).$ Recall the integral $I_p(s)=I_p(\mathcal L; s)$ in \eqref{Ips}. By using the exact same calculations as in \cite[p. 22]{PollackG2}, we find that
\[
I_p(s+1)
= \frac{1}{1-z}
\sum_{[h]}L(hv)|\det(h)|^{-2}z^{\text{val}(\det(h))-c(h)}\big(1-z^{c(h)+1}\big)\charf\big(c(h) \geq 0\big).
\]
We define
\[
P_h(z):=z^{\text{val}(\det(h))-c(h)}\big(1-z^{c(h)+1}\big)\charf\big(c(h) \geq 0\big),
\]
and write
\[
I_p(s+1)= \frac{1}{1-z}\sum_{[h] }L(hv)|\det(h)|^{-2}P_h(z).
\]

To evaluate of $I_p(s)$ in terms of $L$-functions, we must apply $M(\pi_p,s)$ to $I_p(s+1)$ (see \cite[Section 5.4, 5.8]{PollackG2}). The computations follow line-by-line just as in \emph{loc cit}. To demonstrate that the results in \cite{PollackG2} also hold for $p=2,3$, we fill in various details that were omitted in that paper.

Suppose $h = p^ch_0$, with $f_0=f_{\textrm{max}}\cdot h_0$ integral and not divisible by $p$, so that the content of $h$ is $c$, so $c(h) = c$. We begin by explaining the proof of the following lemma, which is a restatement of Lemma 5.10 in \cite{PollackG2}.

\begin{lemma}\label{Lemma5.10inPollackG2}
Suppose $f = p^cf_0$, with $f_0$ in each of the cases enumerated below. Denote by $\Lambda_f$ the
rank two $\mathcal{O}$-lattice corresponding to $f$. Depending on some cases, the content $c(\Lambda')$ of the index $p$ sublattices $\Lambda'$ of $\Lambda_f$ can be described as follows.
\begin{enumerate}
     \item If $f_0$ irreducible mod $p$, then there are $(p + 1)$ sublattices $\Lambda'$ each of which have index $p$ and satisfy $c(\Lambda')= c-1$.
      \item If $f_0 =\ell g$ where $\ell$ is a line and $g$ is irreducible modulo $p$, then there is one sublattice $\Lambda'=\Lambda_{\ell}$ with $c(\Lambda_{\ell}) = c$ while the other $p$ sublattices satisfy $c(\Lambda')=c-1$.
      \item
      If $f_0 =\ell_1\ell_2\ell_3$ where the $\ell_i$ are distinct lines modulo $p$, then there are three sublattices given by $\Lambda'=\Lambda_{\ell_i}$ for $i=1,2, 3$, and each satisfy $c(\Lambda_{\ell_i}) = c$, while for other $(p-2)$ sublattices $\Lambda'$ we have $c(\Lambda')=c-1$.
     \item
     If $f_0 =\ell_1^2\ell_2$ where $\ell_1,\ell_2$ are distinct lines modulo $p$, then there is one sublattice $\Lambda'=\Lambda_{\ell_1}$ with $c(\Lambda_{\ell_1}) = c+1$, another sublattice  $\Lambda'=\Lambda_{\ell_2}$ that has $c(\Lambda_{\ell_2}) = c$  while the other $(p-1)$ sublattices $\Lambda'$ all satisfy $c(\Lambda')=c-1$.
     \item
     If $f_0 =\alpha\ell^3$ where $\ell$ is a line modulo $p$ and $\alpha \in (\mathcal{O}/p)^{\times}$, then there is one sublattice $\Lambda'=\Lambda_{\ell}$ with $c(\Lambda_{\ell}) = c+2$ while the other $p$ sublattices $\Lambda'$ all satisfy $c(\Lambda')=c-1$.
\end{enumerate}
\end{lemma}

\begin{proof}
First we claim that $f_0$ factors into linear factors over an unramified field extension $L$ of $\Q_p$. To see this, let $T_0$ be the cubic ring corresponding to $f_0$. Then by assumption, $T_0 \otimes \Q_p$ is an unramified extension of $\Q_p$. It follows that for some unramified field extension $L/\Q_p$, one has $T_0 \otimes L \approx L \times L \times L$. The binary cubic corresponding to the right-hand side is split, say $xy(x+y)$. The association between binary cubics and cubic rings is clearly compatible with base change, so it follows that $f_0$ factors over $L$.
	
Let $\mathcal{O}_L$ denote the ring of integers in $L$. By Gauss's Lemma, we conclude that $f_0$ factors into linear factors over $\mathcal{O}_L$, say $f_0 = \ell_1 \cdot \ell_2 \cdot \ell_3$. By using this factorization of $f_0$ and the fact that $p$ is a uniformizer in $\mathcal{O}_L$, the lemma follows without much difficulty.
	
Suppose that we are in the final case, so $f_0 \equiv \alpha \ell^3$ modulo $p$. Without loss of generality, we can assume that $\ell = x$. Then by the factorization of $f_0$ given above, we can write $f_0(w,z) = \beta \ell_1'\ell_2'\ell_3'$ with $\beta \in \mathcal{O}_L^\times$ and $\ell_1'\equiv \ell_2'\equiv \ell_3' \equiv x$ modulo $p$.  It follows that $\frac{1}{p}f_0(p(w,z))$ has content $2$, showing that one of the $(q+1)$ sublattices has content $2$.

The rest of the proof proceeds similarly.
\end{proof}

Next, we explain some of the aspects of the proof that were omitted in the explanations after Lemma 5.10 in \cite{PollackG2}. In \emph{loc cit}, we have
\[
P_h(z) = z^{v-c}(1-z^{c+1})\charf\big(c \geq 0\big).
\]
as we set $v=v(h)=\text{val}(\det(h))$ and $c=c(h)$. Define
\[
\mathcal{T}(p) = \GL_2(\mathcal{O})(\begin{smallmatrix}p&0\\0&1 \end{smallmatrix})\GL_2(\mathcal{O})
\]
and
\[
\mathcal{T}(p^{-1}) = \GL_2(\mathcal{O})(\begin{smallmatrix}p^{-1}&0\\0&1 \end{smallmatrix})\GL_2(\mathcal{O}).
\]

From Section 5.7 in \cite{PollackG2}, we recall the function
\[
M_h(z) = p^2P_{hp}(z) + P_{hp^{-1}}(z) + (N(f_{\text{max}}\cdot h)-1)P_h(z) + p P_{h * \mathcal{T}(p)}(z) + P_{h * \mathcal{T}(p^{-1})}(z).
\]
Also recall that for a Hecke operator $Y$ on $\GL_2$ with coset decomposition $Y=\sum_{i} a_i[y_i \GL_2(\mathcal{O})]$, we have
$P_{h*Y}(z)=\sum_{i} a_iP_{hy_i}$. Note that $h y_i$ and  $hp$ are simply scalar multiples of the matrix $h$. Moreover, we define
 \[
 B_0(z) = 1 + (p+1)z +pz^2 + (p^2+p)z^3+p^2 z^4.
 \]
In \cite{PollackG2}, the purpose of the discussion below Lemma 5.10 is to prove the following result, which is a restatement of Lemma 5.12  of \emph{loc cit}.

\begin{lemma}\label{Lemma5.12inPollackG2}
Let the notations be as above. Then
\begin{equation*}
   (1+pz)^{-1}L(E,s)\left(B_0(z)P_h(z) - z^2 M_h(z)\right)=\charf\big({\rm val}(h)=c(h)\big) z^{v-c}\big(1+N(f_{\operatorname{max}})-\epsilon(h_0)z\big).
\end{equation*}
\end{lemma}
We first elaborate on the evaluation of $B_0(z)P_h(z) - z^2 M_h(z)$ in case $c(h) \geq 2$, which is the case explained in \emph{loc cit}. Let
\[
g(z) = p^2 z^{v-c+1}(1-z^{c+2})+z^{v-c-1}(1-z^c)+pz^{v-c}(1-z^{c+1}).
\]
Then, when $c \geq 2$, the first three terms in the above expression for $M_h(z)$ give $g(z)$. The point is that when $h$ is changed, $v$ and $c$ change, as these depend on $h$. One has $v(hp) = v(h) + 2$ and $c(hp) = c(h) +1$. Thus  $P_{hp}(z) = z^{v-c+1}(1-z^{c+2})$. Similarly, $P_{hp^{-1}}(z) = z^{v-c-1}(1-z^c)$, because $v(hp^{-1}) = v(h) -2$ and $c(hp^{-1}) = c(h)-1$. Finally, because $c \geq 1$, we have $N = N(f_{\textrm{max}} \cdot h) = p+1$, so $(N-1)P_{h}(z) = p z^{v-c}(1-z^{c+1})$. Putting these computations together gives
$$
p^2P_{hp}(z) + P_{hp^{-1}}(z) + (N-1)P_h(z)  = g(z).
$$

The terms in $M_h(z)$ with the Hecke operators are computed using Lemma~\ref{Lemma5.10inPollackG2}. For example, when $f_0$ is irreducible modulo $p$, we have $v(h g_i) = v(h) + 1$ and $c(hg_i) = c(h) -1$ using Lemma 5.10. Here $g_i$ are the coset representatives for the Hecke operator $\mathcal{T}(p)$. Then, in this case, $P_{h *\mathcal{T}(p)}(z) = (p+1) z^{v-c+2}(1-z^{c})$.
Similarly, $$P_{h *\mathcal{T}(p^{-1})}(z) = (p+1)z^{v-c+1}(1-z^{c-1}).$$

Combining the above expressions gives the expression for $M_h(z)$ at the bottom of~\cite[p. 27]{PollackG2}.

\begin{remark}
There is a typo on page 27 in~\cite{PollackG2}. In the case $f_0 = \alpha \ell^3$, the term $z^{v-c+2}(1-z^{c+2})$ should instead say $z^{v-c-2}(1-z^{c+2})$.
\end{remark}

We now claim that the cases where $c=1$ in fact do not need to be considered separately from those where $c \geq 2$. Indeed, this is because the terms that vanish because of the $\charf(c(h) \geq 0)$ in case $c=1$ all have a $(1-z^{c-1})$ in them, and so vanish anyway.

We now explain the calculation of $B_0(z)P_h(z) - z^2 M_h(z)$ in the case $c=0$. First note that, in the case $c=0$,
\[
N = N(f_{\textrm{max}}\cdot h) =
   \begin{cases}
        0 & \mbox{ if } \, c=0  \quad \text{and $f_0$ is irreducible modulo } p, \\
        1 & \mbox{ if } \, c=0 \quad  \text{and }\,  f_0= \ell p, \\
        3 & \mbox{ if } \, c=0 \quad  \text{and } \, f_0= \ell_1\ell_2\ell_3,\\
        2 & \mbox{ if }  \, c=0 \quad  \text{and } \, f_0= \ell_1^2\ell_2, \\
        1 & \mbox{ if } \, c=0 \quad  \text{and }  f_0= \alpha\ell^3.
    \end{cases}
\]
From this expression for $N$, one computes $M_h(z)$ in case $c=0$ as follows.
\begin{enumerate}
\setlength\itemsep{0.3em}
	\item Let $f_0$ be irreducible modulo $p$. In this case, we necessarily have $h=1$ and $M_h(z) = p^2z (1-z^2) + (-1)(1-z)$. Also, $P_h(z) = 1-z$ and then $B_0(z)P_h(z) - z^2 M_h(z) = (1+qz)(1-z^3)$.
	\item Let $f_0 \equiv \ell q \pmod{p}$. We have $M_h(z) = p^2 z^{v+1}(1-z^2) + z^v(1-z) + pz^{v+1}(1-z)$ and thus $B_0(z)P_h(z) - z^2 M_h(z) = z^v(1+pz)(1-z^2)$.
	\item Let $f_0 \equiv \ell_1 \ell_2 \ell_3 \pmod{p}$. We have $M_h(z) = p^2 z^{v+1}(1-z^2) + 2z^v(1-z) + 3 p z^{v+1}(1-z)$ and then $B_0(z)P_h(z) - z^2 M_h(z) = z^v(1+pz)(1-z)^2(1+2z)$.
	\item Let $f_0 \equiv \ell_1^2 \ell_2 \pmod{p}$. Then $M_h(z) = p^2 z^{v+1}(1-z^2) + z^v(1-z) + pz^{v+1}(1-z) + pz^v(1-z^2) + z^{v-1}(1-z)$ and $B_0(z)P_h(z) - z^2 M_h(z) =z^v(1+pz)(1-z)(1-z^2)$.
	\item Let $f_0 \equiv \alpha \ell^3 \pmod{p}$. Then $M_h(z) = p^2 z^{v+1}(1-z^2) + pz^{v-1}(1-z^3)+z^{v-2}(1-z^2)$ and $B_0(z)P_h(z) - z^2 M_h(z) = 0$.
\end{enumerate}
Now, for $B_0(z)P_h(z) - z^2 M_h(z)$, we find the same expression as the one on the top of page 28 in \emph{loc cit} except with $c=0$. Hence we obtain Lemma~\ref{Lemma5.12inPollackG2}. Combined with the relationship
\[
(B_0(z)-pz^2\mathcal{T})I(s+1)
= \frac{1}{1-z}\sum_{[h]}L\big(m(h)v\big)\, |\det(h)|^{-2}\big(B_0(z)P_h(z) - z^2 M_h(z)\big),
\]
this completes our evaluation of $I(s)$.

\section{The Eisenstein series}\label{sec:Eis}
We repeat the definition of the normalized Eisenstein series in \eqref{eqn:NEis}.
\[
E^*_\ell(g,s) = \Lambda(s-1)^2 \Lambda(s) \Lambda(2s-4) \frac{\Gamma(s+\ell-1)\Gamma(s+\ell-2)}{\Gamma(s-1)\Gamma(s-2)} E_\ell(g,s).
\]
The purpose of this section is to prove the following theorem.

\begin{theorem}\label{thm:Eis}
We have
\[
E^*_\ell(g,s) = E^*_\ell(g,5-s).
\]
\end{theorem}

This theorem is an immediate consequence of Langlands' functional equation, with the difficulty lying in the computation of the appropriate intertwining operator of the section $f_{\ell}(g,s)$.  We remark that Segal \cite{segalEis} has studied the poles of this and related Eisenstein series in a right half-plane.

Consider the diagonal maximal $T'$ of $G'$ consisting of the elements
\[
t=\diag\big(t_1,t_2,t_3,t_4,t_4^{-1},t_3^{-1},t_2^{-1},t_1^{-1}\big).
\]
For $1 \leq j \leq 4$, let $r_j'$ denote the characters of $T'$ that takes the element $t$ to $t_j$. We fix a maximal  $T$ of $G$ that maps to $T'$ under the map $G \rightarrow G'$, and write $r_j$ for the restriction of $r_j'$ to $T$.  We label the Dynkin diagram of $G$ by roots $\alpha_1 = r_1-r_2,\alpha_2 = r_3+r_4,\alpha_3 = r_3-r_4,\alpha_4 = r_2-r_3$. Then $\alpha_4$ is the central vertex of the diagram.

We abuse notation and also denote the restriction to $T$ of the character $t \mapsto |t_j|$ of $T'$ by $r_j$. Then the inducing character for our Eisenstein series is $|\nu|^s = s(r_1 + r_2)$. This is in $\mathrm{Ind}_{B}^{G}(\delta_B^{1/2}\lambda_s)$ with $\delta_B^{1/2} = 3r_1 + 2r_2 + r_3$ and $\lambda_s = (s-3)r_1 + (s-2)r_2 -r_3$.

Let $N$ be the unipotent radical of the Heisenberg parabolic, so that the roots in $N$ are $r_1-r_3, r_1-r_4,r_1+r_4, r_1+r_3, r_2-r_3,r_2-r_4,r_2+r_4,r_2+r_3, r_1 +r_2$. The long intertwiner for $N$ is $w=[412343214] = [412434214]$. Here the notation $[ijk]$ means that one performs a reflection in the roots $i,j,k$ from right to left. To see that this expression for $w$ as a product of simple reflections is correct, one checks that $w$ makes the roots in $N$ negative, and that it has length $9$.

Now, we set
\[
M(w,s)f_{\ell}(g,s)= \int_{N(\A)}{f_{\ell}(w^{-1}ng,s)\,dn}.
\]
If the real part of $s$ is sufficiently large, then this integral converges and has a meromorphic continuation in $s$. We will prove the following result on the integral.

\begin{proposition}\label{prop:intertwine}
We have
\[
M(w,s) f_{\ell}(g,s) = c_{\ell}(s) f_{\ell}(g,5-s),
\]
where
	\[
	c_{\ell}(s)
	= \frac{\Lambda(s-3)^2 \Lambda(s-4)\Lambda(2s-5)}{\Lambda(s-1)^2\Lambda(s)\Lambda(2s-4)}\frac{\Gamma(s-2)\Gamma(s-3)\Gamma(s-2)\Gamma(s-1)}{\Gamma(s-\ell-3)\Gamma(s-\ell-2)\Gamma(s+\ell-1)\Gamma(s+\ell-2)}.
	\]
\end{proposition}

\begin{proof}[Proof of Theorem \ref{thm:Eis}]
We note the identity
\[
\frac{\Gamma(s-2)\Gamma(s-3)}{\Gamma(s-\ell-2)\Gamma(s-\ell-3)} = \frac{\Gamma(4-s+\ell)\Gamma(3-s+\ell)}{\Gamma(4-s)\Gamma(3-s)}.
\]
The result then follows from Proposition \ref{prop:intertwine}, Definition \eqref{eqn:NEis} and Langlands' functional equation.
\end{proof}

For the rest of this section, we focus on proving Proposition \ref{prop:intertwine}. We first introduce some notation. Let $x,y$ be indeterminates, $s$ be a complex parameter, and $\ell$ be a fixed positive even integer. Also, set
\[
f_{+} = x+y \quad \text{and} \quad f_{-} = x-y.
\]
We have
\begin{align*}
\textrm{Span}\big(\big\{x^{2\ell-2j}y^{2j} + x^{2j}y^{2\ell-2j}: &\, 0 \leq j \leq \ell/2\big\}\big)  \\
= \textrm{Span} &\big(\big\{ f_{+}^{2\ell-2j} f_{-}^{2j}+f_{+}^{2j}f_{-}^{2\ell-2j}: 0 \leq j \leq \ell/2\big\}\big):=V_{\textrm{even}},
\end{align*}
say. We think of $V_{\textrm{even}}$ as sitting inside the space $\mathbf{V}_{\ell} = \textrm{Sym}^{2\ell}(\C^2)$ (see Section \ref{subsec:Eis}). We will define a few operators on the space $V_{\textrm{even}}$. For a nonnegative integer $k$ and $z \in \C$, let
\[
(z)_k = z(z+1)(z+2) \cdots (z+k-1).
\]
This is the so-called Pochhammer symbol.

For a complex number $s$, we define $[s;x,y]$ as the diagonal operator on $V_{\textrm{even}}$ given by
\[
x^{2\ell-2j}y^{2j}+x^{2j}y^{2\ell-2j} \mapsto \frac{\left(\frac{1-s}{2}\right)_{|\frac{\ell}{2}-j|}}{\left(\frac{1+s}{2}\right)_{|\frac{\ell}{2}-j|}}\, \left(x^{2\ell-2j} y^{2j}+x^{2j}y^{2\ell-2j}\right)
\]
for each $0 \leq j \leq \ell/2$. Similarly, we define $[s;f_{+},f_{-}]$ as the diagonal operator on $V_{\textrm{even}}$ that is given by
\[
f_{+}^{2\ell-2j}f_{-}^{2j} + f_{+}^{2j}f_{-}^{2\ell-2j}  \mapsto \frac{\left(\frac{1-s}{2}\right)_{|\frac{\ell}{2}-j|}}{\left(\frac{1+s}{2}\right)_{|\frac{\ell}{2}-j|}}\, \left(f_{+}^{2\ell-2j} f_{-}^{2j}+f_{+}^{2j}f_{-}^{2\ell-2j}\right)
\]
for each $0 \leq j \leq \ell/2$.

\begin{proof}[Proof of Proposition \ref{prop:intertwine}] \let\qedsymbol\relax
In our computation of the image of $f_{\ell}(g,s)$ under the intertwining operator $M(w,s)$, we will use the so-called cocycle property which was given in Theorem 4.2.2 in \cite{shahidiBook}. This theorem implies that $M(w,s)$ can be viewed as a composition of intertwiners that are associated to simple reflections.

To apply the cocycle property, we record how the simple reflections in the product $w=[412434214]$ move the character $\lambda_s = (s-3)r_1 + (s-2)r_2 + (-1)r_3$ around, and how the associated one-dimensional intertwining operators act on the inducing section $f_{\ell}(g,s).$ This is provided in Table~\ref{table:Inter}.

\begin{table}[h]
\caption{Intertwining operators}\label{table:Inter}
\bgroup
\def\arraystretch{2.3}
\begin{tabular}{|c|c|c|}
	\hline
Simple reflection & Intertwiner & New character \\ \hline
$[4]$ & $\displaystyle \frac{\Lambda(s-1)}{\Lambda(s)}[s-1;f_{+},f_{-}]$ & $(s-3)r_1 + (-1)r_2 + (s-2)r_3$ \\ \hline
$[1]$ & $\displaystyle \frac{\Lambda(s-2)}{\Lambda(s-1)}[s-2;x,y]$ & $(-1)r_1 + (s-3)r_2 + (s-2)r_3$\\ \hline
$[2]$ & $\displaystyle \frac{\Lambda(s-2)}{\Lambda(s-1)}[s-2;x,y]$ & $(-1)r_1 + (s-3)r_2 + (2-s)r_4$\\ \hline
$[4]$ & $\displaystyle \frac{\Lambda(s-3)}{\Lambda(s-2)}[s-3;f_{+},f_{-}]$ & $(-1)r_1 + (s-3)r_3 + (2-s)r_4$ \\ \hline
$[3]$ & $\displaystyle \frac{\Lambda(2s-5)}{\Lambda(2s-4)}[2s-5;x,y]$ & $(-1)r_1 + (2-s)r_3 + (s-3)r_4$\\ \hline
$[4]$ & $\displaystyle \frac{\Lambda(s-2)}{\Lambda(s-1)}[s-2;f_{+},f_{-}]$ & $(-1)r_1 + (2-s)r_2 + (s-3)r_4$\\ \hline
$[2]$ & $\displaystyle \frac{\Lambda(s-3)}{\Lambda(s-2)}[s-3;x,y]$ & $(-1)r_1 + (2-s)r_2 + (3-s)r_3$\\ \hline
$[1]$ & $\displaystyle \frac{\Lambda(s-3)}{\Lambda(s-2)}[s-3;x,y]$ & $(2-s)r_1 + (-1)r_2 + (3-s)r_3$\\ \hline
$[4]$ & $\displaystyle \frac{\Lambda(s-4)}{\Lambda(s-3)}[s-4;f_{+},f_{-}]$ & $(2-s)r_1 + (3-s)r_2 + (-1)r_3$\\ \hline
\end{tabular}
\egroup
\end{table}

The notation in this table has the following meaning, as we explain by an example. Set
\[
\lambda_s' = (s-3)r_1 + (-1)r_2 + (s-2)r_3.
\]
When we apply the intertwining operator associated to the reflection $[4]$ to the inducing section $f_{\ell}(g,s)$, we obtain the unique $K_f$-spherical, $K_\infty$-equivariant element of $\mathrm{Ind}(\delta_B^{1/2} \lambda_s')$, whose value at $g=1$ is
\[
\frac{\Lambda(s-1)}{\Lambda(s)}[s-1;f_{+}, f_{-}](x^{\ell} y^{\ell}).
\]
Denote the resulting inducing section by $f_{\ell}'(g,s)$, and set
\[
\lambda_{s}'' = (-1)r_1 + (s-3)r_2 + (s-2)r_3.
\]
Then, if we apply the intertwining operator associated to the reflection $[1]$ to $f_{\ell}'(g,s)$, then we obtain the unique $K_f$-spherical, $K_\infty$-equivariant element of $\mathrm{Ind}(\delta_B^{1/2} \lambda_s'')$, whose value at $g=1$ is
\[
\bigg(\frac{\Lambda(s-2)}{\Lambda(s-1)}[s-2;x,y]\circ \frac{\Lambda(s-1)}{\Lambda(s)}[s-1;f_{+}, f_{-}]\bigg)(x^{\ell} y^{\ell}).
\]

Note that the terms in Table \ref{table:Inter} that are in the form of ratios of $\Lambda$-values follow from a formula of Gindikin and Karpelevich, while the terms such as $[s-1;f_{+},f_{-}]$ and $[s-2;x,y]$ arise due to the fact that our inducing section is not spherical at the archimedean place. We postpone providing a complete justification of the operators in Table \ref{table:Inter} until the next section. Granted this, the $\Lambda$-values multiply to
\[
  \frac{\Lambda(s-3)^2 \Lambda(s-4)\Lambda(2s-5)}{\Lambda(s-1)^2\Lambda(s)\Lambda(2s-4)}
\]
(see also Table 12 in \cite{segalEis}). The other terms give the polynomial intertwiner
\begin{equation}
\begin{split}\label{eqn:Mpoly}
  M_{\textrm{poly}}(s)
  = &\, [s-4;f_{+},f_{-}] \circ [s-3;x,y]^2 \circ [s-2;f_{+},f_{-}] \circ [2s-5;x,y]  \\
  & \circ [s-3;f_{+},f_{-}] \circ [s-2;x,y]^2 \circ [s-1;f_{+},f_{-}].
\end{split}
\end{equation}
The proposition now follows from the following proposition.
\end{proof}

\begin{proposition}\label{prop:Polyintertwine}
Let $M_{\operatorname{poly}}(s)$ be as in \eqref{eqn:Mpoly}. We have
\[
  M_{\operatorname{poly}}(s) x^{\ell} y^{\ell} = c_{\operatorname{poly},\ell}(s) x^{\ell} y^{\ell},
\]
where
	\begin{align*} c_{\operatorname{poly},\ell}(s)
	&= \frac{(s-3) (s-4)^2 (s-5)^2 \cdots (s-\ell-2)^2 (s-\ell-3)}{(s+\ell-2)(s+\ell-3)^2 (s+\ell-4)^2 \cdots (s-1)^2 (s-2)} \\ &=\frac{\Gamma(s-2)\Gamma(s-3)\Gamma(s-2)\Gamma(s-1)}{\Gamma(s-\ell-3)\Gamma(s-\ell-2)\Gamma(s+\ell-1)\Gamma(s+\ell-2)}. 	\end{align*}
\end{proposition}
In other words, the above proposition proves that $x^{\ell}y^{\ell}$ is an eigenvector for the operator $M_{\textrm{poly}}(s)$ with the eigenvalue $c_{\textrm{poly},\ell}(s)$. This proposition will be proved in Section \ref{subsec:polyInter}.

\subsection{The root intertwiners}
The purpose of this section is to explain the presence of the terms that appear in the ``Intertwiner" column of Table \ref{table:Inter}. We require the following lemma.

\begin{lemma}\label{lem:SL2inter}
Let $B$ denote the upper-triangular Borel of $\SL_2$. For $\theta \in \R$, set $k_\theta = \mm{\cos{\theta}}{-\sin{\theta}}{\sin{\theta}}{\cos{\theta}}$. Suppose that $f_{\SL_2,j}(g,s)$ is the section in $\mathrm{Ind}_{B(\R)}^{\SL_2(\R)}(\delta_B^{1/2}\delta_B^{s/2})$ that satisfies
\[
f_{\SL_2,j}\big(gk_\theta,s\big) = e^{i j \theta} f_{\SL_2,j}(g,s)
\]
for all $g \in \SL_2(\R)$ and $k_\theta \in \SO(2)$ as above. Then
	\[
\int_{\R}	f_{\SL_2,j}\bigg(
  \Big(
  \begin{array}{cc}{0}&{-1}\\{1}&{0}\end{array}\Big)
  \Big(\begin{array}{cc}{1}&{x}\\{0}&{1}
  \end{array}\Big)g, s
  \bigg) \, dx = i^{j} \frac{\Gamma_\C(s)}{\Gamma_\R(s-j+1)\Gamma_\R(s+j+1)} f_{\SL_2,j}(g,-s).
	\]
If $j$ is even, then this becomes
	\[
\int_{\R}	f_{\SL_2,j}\bigg(
  \Big(
  \begin{array}{cc}{0}&{-1}\\{1}&{0}\end{array}\Big)
  \Big(\begin{array}{cc}{1}&{x}\\{0}&{1}
  \end{array}\Big)g, s
  \bigg) \, dx =\frac{\Gamma_\R(s)}{\Gamma_\R(s+1)} \frac{\left(\frac{1-s}{2}\right)_{|j/2|}}{\left(\frac{1+s}{2}\right)_{|j/2|}} f_{\SL_2,j}(g,-s).
	\]
\end{lemma}
\begin{proof}
The proof is standard. To give some details anyway, we consider the case $g=1$. Note that
	\[
	\mm{0}{-1}{1}{0} \mm{1}{x}{0}{1}
	= \mm{1}{-x(x^2+1)^{-1}}{0}{1} \mm{(x^2+1)^{-1/2}}{0}{0}{(x^2+1)^{1/2}}\mm{x (x^2+1)^{-1/2}}{- (x^2+1)^{-1/2}}{ (x^2+1)^{-1/2}}{x  (x^2+1)^{-1/2}}.
	\]
From this, we obtain
	\begin{align*}
\int_{\R}	f_{\SL_2,j}\bigg(
  \Big(
  \begin{array}{cc}{0}&{-1}\\{1}&{0}\end{array}\Big)
  \Big(\begin{array}{cc}{1}&{x}\\{0}&{1}
  \end{array}\Big), s
  \bigg) \, dx
	&= \int_{\R}{(x^2+1)^{-(s+1)/2} \bigg(\frac{x+i}{(x^2+1)^{1/2}}\bigg)^{j}\,dx} \\
	&= \int_{\R}{(x+i)^{-(s-j+1)/2} (x-i)^{-(s+j+1)/2}\,dx}.
	\end{align*}
This last integral is evaluated in \cite[p. 279]{miyake}, which gives
	\[
	i^{j} 2^{1-s}\pi \frac{\Gamma(s)}{\Gamma(\frac{s-j+1}{2})\Gamma(\frac{s+j+1}{2})} = i^{j} \frac{\Gamma_\C(s)}{\Gamma_\R(s-j+1)\Gamma_\R(s+j+1)}.
	\]
Note that when $j$ is even,
	\begin{align*}
	\frac{\Gamma_\C(s)}{\Gamma_\R(s-j+1)\Gamma_\R(s+j+1)}
	&= \frac{\Gamma_\R(s)\Gamma_\R(s+1)}{\Gamma_\R(s-j+1)\Gamma_\R(s+j+1)}\\
	&= \frac{\Gamma_\R(s)}{\Gamma_\R(s+1)} \frac{\Gamma_\R(s+1)\Gamma_\R(s+1)}{\Gamma_\R(s-j+1)\Gamma_\R(s+j+1)},
\end{align*}
	and
	\begin{align*}
	&\frac{\Gamma_\R(s+1)\Gamma_\R(s+1)}{\Gamma_\R(s-j+1)\Gamma_\R(s+j+1)} \\
	&\quad = \frac{\Gamma((s+1)/2)^2}{\Gamma((s-j+1)/2)\Gamma((s+j+1)/2)}
	= \frac{\left(\frac{s+1-|j|}{2}\right)_{|j/2|}}{\left(\frac{s+1}{2}\right)_{|j/2|}}
	= (-1)^{j/2} \frac{\left(\frac{1-s}{2}\right)_{|j/2|}}{\left(\frac{1+s}{2}\right)_{|j/2|}}.
\end{align*}
The result then follows.
\end{proof}

Before applying the above lemma, we note the following calculations. For each positive root $\alpha$, let $\varphi_{\alpha}: \SL_2 \rightarrow G$ be the root $\SL_2$ that is determined by a pinning of $G$. This pinning is assumed to be compatible with the Cartan involutions, that is, $\theta(\varphi_{\alpha}(g)) = \varphi_{\alpha}(\,^{t}g^{-1})$.  On the Lie algebra level, the root $\sl_2$'s give rise to the elements $d\varphi_{\alpha}\left(\mm{0}{1}{-1}{0}\right)$ in the Lie algebra of $G$.  We list these elements now.

\begin{itemize}
\setlength\itemsep{0.3em}
	\item $\displaystyle \alpha_1 = r_1 - r_2, \, b_1 \wedge b_{-2} + b_{-1} \wedge b_2 = u_1 \wedge u_2 - u_{-1} \wedge u_{-2} = -\frac{i}{2}(h^{+}+h'^{+} + h^{-} + h'^{-})$,
	\item $\displaystyle \alpha_2 = r_3+r_4, \, b_3 \wedge b_4 + b_{-3} \wedge b_{-4} = v_1 \wedge v_2 + v_{-1} \wedge v_{-2} = -\frac{i}{2}(h^{+}-h'^{+} - h^{-} + h'^{-})$,
	\item $\displaystyle \alpha_3 = r_3 -r_4, \, b_3 \wedge b_{-4} + b_{-3} \wedge b_4 = v_1 \wedge v_2 - v_{-1} \wedge v_{-2} = -\frac{i}{2}(h^{+}-h'^{+} + h^{-} - h'^{-})$,
	\item $\displaystyle \alpha_4 = r_2-r_3, \, b_2 \wedge b_{-3} + b_{-2} \wedge b_{3} = u_2 \wedge v_1 - u_{-2} \wedge v_{-1} = -\frac{i}{2}(-H^+-H'^+ +H^{-}+ H'^{-})$ where $H^{?} = e^{?}+ f^{?}$ for $? \in \{+,\,'+,-,\,'-\}$.
\end{itemize}

By combining these calculations with Lemma \ref{lem:SL2inter}, we arrive at the proposition below. Let $w_j \in N(T)_{\Q} \cap (K_f K)$ be the simple reflection in the Weyl group that corresponds to the simple root $\alpha_j$. The proposition computes the rank one intertwining operator associated to $w_j$ on the inducing sections that arise in Table \ref{table:Inter}. Below, $U_{\alpha_j}$ denotes the unipotent subgroup of $G$ that corresponds to the simple root $\alpha_j$.

\begin{proposition}
Suppose that $\lambda = \alpha_1 r_1 + \alpha_2 r_2 + \alpha_3 r_3 + \alpha_4 r_4$ is an unramified character of $T(\A)$. Let $f \in \operatorname{Ind}_{B(\A)}^{G(\A)}(\delta_B^{1/2} \lambda)$ be $\mathbf{V}_{\ell}$-valued such that $f$ is $K$-equivariant, $K_f$-invariant, and $f(1) \in V_{\operatorname{even}}$. Set $s = \langle \alpha_j^\vee, \lambda \rangle$, where $\alpha_j^\vee$ is the coroot associated to $\alpha_j$. If $s >1$, then the integral
	\[
	M(w_j)f(g) = \int_{U_{\alpha_j}(\A)}{f(w_j^{-1}xg)\,dx}
	\]
is absolutely convergent. Moreover, the value $M(w_j)f(g)$ is the unique $\mathbf{V}_{\ell}$-valued, $K$-equivariant, $K_f$-invariant element of $\operatorname{Ind}_{B(\A)}^{G(\A)}(\delta_B^{1/2}w_j(\lambda))$, and its value at $g=1$ is
\[
M(w_j)f(1) =
\begin{dcases} \frac{\Lambda(s)}{\Lambda(s+1)}[s;x,y]f(1) &\text{ if } \,\, j =1,2,3, \\ \frac{\Lambda(s)}{\Lambda(s+1)}[s;f_{+},f_{-}]f(1)&\text{ if } \,\, j = 4.
\end{dcases}
\]
\end{proposition}

\begin{proof}
The proof is standard except the computation of the value of $M(w_j)f(g)$ at $g=1$.  We explain the case $j=1$ as the other cases are similar.  To evaluate $M(w_1)f(1)$, we use a pinning of $G$ to pull back the calculation to $\SL_2$. Thus, we assume that $\varphi_{\alpha_1}: \SL_2 \rightarrow G$ is compatible with the integral structures, the Cartan involution, and that it satisfies the properties
\[
  \varphi_{1}\Big(
  \begin{array}{cc}{0}&{1}\\{-1}&{0}\end{array}\Big) = w_1
  \quad \text{and}
  \quad \varphi_{1}
  \Big(\begin{array}{cc}{1}&{*}\\{0}&{1}
  \end{array}\Big) = U_{\alpha_1}.
\]
We can also assume that
  \[
  d\varphi_1\Big(
  \begin{array}{cc}{0}&{1}\\{-1}&{0}\end{array}\Big)  = b_1 \wedge b_{-2} + b_{-1} \wedge b_2.
  \]
Then we have
\[
  M(w_1)f(1)
  = \int_{\A}{f\bigg(\varphi_1
  \Big(
  \begin{array}{cc}{0}&{1}\\{-1}&{0}\end{array}\Big)
  \Big(\begin{array}{cc}{1}&{x}\\{0}&{1}
  \end{array}\Big)
  \bigg)\,dx}.
\]
The function $f \circ \varphi_{1}$ on $\SL_2$ is an element of the induction space corresponding to $\delta_B^{(s+1)/2}$, as in Lemma \ref{lem:SL2inter}. The integrals over the finite places give $\frac{\zeta(s)}{\zeta(s+1)}$, as standard. The computation of the integral over $\R$ follows from the fact that $d\varphi_1(\mm{0}{1}{-1}{0})$ acts on $x^{2\ell-2j}y^{2j}$ the same way as $-\frac{i}{2} h_{\ell}^{-}$, which acts by $-i(\ell-2j)$. The result thus follows from Lemma \ref{lem:SL2inter}.
\end{proof}

The above proposition yields the intertwining operators that appear in Table \ref{table:Inter}.


\subsection{The polynomial intertwiner}\label{subsec:polyInter}
The purpose of this section is to prove Proposition \ref{prop:Polyintertwine}, from which Proposition \ref{prop:intertwine} follows. The proof of Proposition \ref{prop:Polyintertwine} requires the following two lemmas.

\begin{lemma}\label{lem:Fuv}
Let $u, v$ be variables and $w$ be a complex parameter. Also, put
  \[
  F_w(u,v) = \big(1-2u-2v + (u-v)^2\big)^{-w}.
  \]
Then
	\[
	F_{w}(u,v) = (1-2u-2v + (u-v)^2)^{-w} = \sum_{j, k \geq 0} p_{j,k}(w) \frac{u^j v^k}{j! k!},
	\]
where
   \[
   p_{j,k}(w) =\frac{\Gamma(2w+j+k)\Gamma(w+j+k+1/2)\Gamma(w+1/2)}{\Gamma(2w)\Gamma(w+k+1/2)\Gamma(w+j+1/2)}.
   \]
\end{lemma}

\begin{proof}
First, note that the $p_{j,k}(w)$ are polynomials in $w.$ To see this, one uses the functional equation for the gamma function, that is, $\Gamma(s+1) = s\Gamma(s)$. Now, if the polynomials $p_{j,k}(w)$ are indeed the Taylor coefficients of $F_w(u,v)$, then these polynomials must satisfy the expression
	\begin{equation}\label{eqn:pjk}
		\begin{split}
	p_{j,k}(w-1)=&\, p_{j,k}(w) - 2j p_{j-1,k}(w) - 2k p_{j,k-1}(w) + j(j-1) p_{j-2,k}(w) \\
	&- 2jkp_{j-1,k-1}(w) + k(k-1) p_{j,k-2}(w) .
	\end{split}
	\end{equation}
This relationship comes from comparing both sides of the identity
  \[
  \big(1-2u-2v + u^2 -2uv + v^2\big) \Big(\sum_{j,k} p_{j,k}(w) \frac{u^j v^k}{j! k!} \Big) = F_{w-1}(u,v) = \sum_{j,k} p_{j,k}(w-1) \frac{u^j v^k}{j! k!}.
  \]
Now, we have two claims:
	\begin{enumerate}[label=(\roman*)]
		\item One can verify \eqref{eqn:pjk} directly.
		\item Combined with the fact that the $p_{j,k}(w)$ are polynomials, \eqref{eqn:pjk} implies the lemma.
	\end{enumerate}

For the proof of the first claim, we again use the functional equation $ \Gamma(s+1)=s \Gamma(s)$ and relate the polynomials $p_{j,k}(w)$ to $p_{j,k}(w-1)$. For example, we obtain
	\[
	p_{j,k}(w) = \frac{(2w-2+j+k)(2w-1+j+k) (w+j+k-1/2)(w-1/2)}{(2w-2)(2w-1)(w+j-1/2)(w+k-1/2)} p_{j,k}(w-1).
	\]
One can obtain similar expressions relating $p_{j-1,k}(w), p_{j-1,k-1}(w),...$ to $p_{j,k}(w-1)$. Thus the identity in \eqref{eqn:pjk} becomes an identity for rational functions of $w, j, k$, which one can then verify directly.
	
For the proof of the second claim, note that the Taylor coefficients of $F_{w}(u,v)$ are necessarily polynomials in $w$. Thus, to see that they are equal to $p_{j,k}(w)$, it suffices to check that they are equal at infinitely many integers. But \eqref{eqn:pjk} allows one to induct, and thus verify the Taylor expansion of $F_w(u,v)$ for all negative integers $w$.

This completes the proof of the lemma.
\end{proof}

\begin{lemma}\label{lem:c's}
Let $f_+=x+y$ and $f_-=x-y$ as before. Then we have
\begin{enumerate}[label=(\roman*)]\setlength\itemsep{0.3em}
	\item $\displaystyle \big([s-1;x,y] \circ [s;f_{+},f_{-}] \big) (x^{\ell} y^{\ell}) = \frac{c'(s)}{2^{\ell}}  f_{+}^{\ell} f_{-}^{\ell}$\, ,
	\item $\displaystyle \big([s-1;f_{+},f_{-}] \circ [s;x,y] \big)(f_{+}^{\ell} f_{-}^{\ell}) = 2^{\ell} c'(s) x^{\ell} y^{\ell}$,
\end{enumerate}
where we set
	\[
	c'(s) = \frac{(s-\ell)(s-\ell+2) \cdots (s-4)(s-2)}{(s+1)(s+3) \cdots (s+\ell-3)(s+\ell-1)}.
	\]
\end{lemma}
\begin{proof}
Part (ii) follows from part (i) by switching the roles of $x,y$ with $f_{+}, f_{-}$. We now prove part (i). Note that $ x = \frac{f_{+}+f_{-}}{2}$ and $y = \frac{f_{+}-f_{-}}{2}$, so $4xy = f_{+}^2-f_{-}^2$. Thus
\begin{equation}\label{eqn:4xy}4^{\ell} \frac{x^{\ell}y^{\ell}}{\ell!} = \sum_{j=0}^{\ell}{(-1)^{j} \frac{f_{+}^{2\ell-2j}f_{-}^{2j}}{(\ell-j)!j!}}.\end{equation}
Also, observe that
\begin{equation}\label{eqn:Dsxy} \frac{\left(\frac{1-s}{2}\right)_{\left|\frac{\ell}{2}-j\right|}}{\left(\frac{1+s}{2}\right)_{\left|\frac{\ell}{2}-j\right|}} = (-1)^{\frac{\ell}{2}-j} \,
\dfrac{\Gamma\left(\frac{1-s}{2}- (\frac{\ell}{2}-j)\right)\Gamma\left(\frac{1-s}{2}+ (\frac{\ell}{2}-j)\right)}{\Gamma\left(\frac{1-s}{2}\right)^2}.
\end{equation}
Let $w = \frac{1-s-\ell}{2}$. Then by combining \eqref{eqn:4xy} and \eqref{eqn:Dsxy}, we obtain
	\[
	(-1)^{\frac{\ell}{2}}\, 4^{\ell}\, [s;f_{+},f_{-}] \frac{x^\ell y^{\ell}}{\ell !} = \frac{\Gamma(w)^2}{\Gamma(w+\ell/2)^2} \sum_{j=0}^{\ell/2} {\frac{\Gamma(w+\ell-j)\Gamma(w+j)}{\Gamma(w)^2} \frac{f_{+}^{2\ell-2j} f_{-}^{2j}}{(\ell-j)! j!}}.
	\]
We sum $\frac{\Gamma(w+\ell/2)^2}{\Gamma(w)^2}$ times the right-hand side over non-negative even integers $\ell$ and use the identity
\[
\sum_{k \geq 0}\frac{\Gamma(w+k)}{\Gamma(w)}\frac{z^k}{k!} = (1-z)^{-w},
\]
to find that
\begin{align*}
  \frac{\Gamma(w)^2}{\Gamma(w+\ell/2)^2} (1-f_{+}^2)^{-w}(1-f_{-}^2)^{-w}
  &= \frac{\Gamma(w)^2}{\Gamma(w+\ell/2)^2} (1-2(x^2+y^2)+(x^2-y^2)^2)^{-w} \\
  &= \frac{\Gamma(w)^2}{\Gamma(w+\ell/2)^2} \sum_{j,k \geq 0}{p_{j,k}(w) \frac{x^{2j}y^{2k}}{j! k!}}.
\end{align*}
In the first equality, we used the relations $f_{+}f_{-} = x^2-y^2$ and $f_{+}^2+f_{-}^2= 2(x^2+y^2)$. The second equality follows from Lemma \ref{lem:Fuv}.

Note that if $j + k = \ell$, then
\[
  [s-1;x,y]\left(x^{2j}y^{2k}+x^{2k}y^{2j}\right) = (-1)^{\ell/2+k} \frac{\Gamma(w + j + \frac{1}{2}) \Gamma(w+k+\frac{1}{2})}{\Gamma(w+\frac{\ell+1}{2})^2} \left(x^{2j}y^{2k}+x^{2k}y^{2j}\right).
\]
Then
\begin{align*}
   4^{\ell}&\left( [s-1;x,y]\circ [s;f_{+},f_{-}] \right) \left(\frac{x^{\ell}y^{\ell}}{\ell!}\right) \\  =\frac{\Gamma(w)^2}{\Gamma(w+\ell/2)^2} \sum_{\substack{j,k \geq 0, \\ j+k = \ell}} & {(-1)^k p_{j,k}(w) \frac{\Gamma(w+j+1/2)\Gamma(w+k+1/2)}{\Gamma(w+(\ell+1)/2)^2} \frac{x^{2j}y^{2k}}{j!k!}}.
\end{align*}
By Lemma \ref{lem:Fuv},
\[
   p_{j,k}(w) \frac{\Gamma(w+j+1/2)\Gamma(w+k+1/2)}{\Gamma(w+(\ell+1)/2)^2} = \frac{\Gamma(2w+\ell)\Gamma(w+\ell+1/2)\Gamma(w+1/2)}{\Gamma(2w)\Gamma(w+(\ell+1)/2)^2}.
\]
Then
\[
4^{\ell} \big([s-1;x,y]\circ [s;f_{+},f_{-}]\big)\left( \frac{x^{\ell} y^{\ell}}{\ell!}\right)
=
   \frac{\Gamma(w)^2\Gamma(2w+\ell)\Gamma(w+\ell+1/2)\Gamma(w+1/2)}{\Gamma(w+\ell/2)^2\Gamma(w+(\ell+1)/2)^2\Gamma(2w)} \frac{(x^2-y^2)^{\ell}}{\ell !}.
\]
Here, $x^2-y^2 = f_{+}f_{-}$. Thus, rewriting the above equation in terms of $s$ gives the statement in the lemma. Indeed, the product of the gamma functions can be written as
\[
  \frac{\Gamma(2w+\ell)}{\Gamma(2w)} \cdot  \frac{\Gamma(w+\ell+1/2)}{\Gamma(w+(\ell+1)/2)} \cdot \frac{\Gamma(w+1/2)}{\Gamma(w+(\ell+1)/2)} \cdot \frac{\Gamma(w)^2}{\Gamma(w+\ell/2)^2},
\]
where each of these individual ratio of gamma functions is a rational function of $w$. We find that this rational function equals
\[
  2^{\ell} \frac{(2w+\ell+1)(2w + \ell + 3)\cdots (2w+2\ell-1)}{(2w)(2w+2)(2w+4) \cdots (2w+\ell-2)} = 2^{\ell} \frac{(s-\ell)(s-\ell+2) \cdots (s-6)(s-4)(s-2)}{(s+1)(s+3) \cdots (s+\ell-3)(s+\ell-1)}.
\]
The above is $2^{\ell} c'(s)$. This completes the proof of part (i) and hence the lemma.
\end{proof}

\begin{proof}[Proof of Proposition \ref{prop:Polyintertwine}]
The proposition now follows easily from the factorization of $M_{\operatorname{poly}}(s)$ in \eqref{eqn:Mpoly} and Lemma \ref{lem:c's}.
\end{proof}


\section{Archimedean zeta integral} \label{sec:archInt}
In this section, we explicitly compute the archimedean integral that is part of the Rankin-Selberg integral. Below, we use the symbol $\sim$ to denote equality up to a nonzero constant that may or may not depend on the weight $\ell$ of the modular form. Also, the constant that is implied by $\sim$ may be different at each occurrence of the symbol.

Recall that in Section~\ref{local integrals} we defined
\begin{equation}\label{eq:norm zeta integral}
I^*(s;\ell) =2^{s} \Gamma_\R(s-1) \Gamma_\C(s+\ell-1)\Gamma_\C(s+\ell-2)  I(s;\ell),
\end{equation}
where
\[
I(s;\ell) = \int_{N^{0,E}(\R)\backslash G_2(\R)}{\{f_{\ell}(\gamma_0 g,s), \mathcal{W}_\chi(g)\}_K\,dg}.
\]
Here $\mathcal{W}_\chi$ is the generalized Whittaker function. This means that $\mathcal{W}_\chi: G_2(\mathbb{R}) \to \mathbf{V}_{\ell}$ is a smooth function of moderate growth which satisfies the condition
\[
 \mathcal{W}_\chi (ngk)=\chi(n) k^{-1} \cdot  \mathcal{W}_\chi(g) \quad \text{for all } \, n\in N(\mathbb{R}), \, k\in K \, \, \text{and} \, \, g\in G_2(\mathbb{R}),
\]
and we have $\mathcal{D}_\ell \mathcal{W}_\chi =0 $ for the Schmid operator $\mathcal{D}_\ell$ (see \cite[p. 10]{PollackG2}). Also, the braces $\{\,,\,\}_K$ denote the $K$-equivariant pairing on $\mathbf{V}_{\ell}$ that is unique up to a scalar multiple.

Our goal is to prove the following theorem.
\begin{theorem}\label{thm:I*int}
We have
	\[
	I^*(s;\ell)
	\sim
	\Gamma_\R(s-1)\Gamma_\C(s+\ell-3)\Gamma_\C(s+\ell-2)\Gamma_\C(s+2\ell-3).
	\]
\end{theorem}
Note that by \eqref{eq:norm zeta integral}, it suffices to compute the integral $I(s;\ell)$. In Section 6 of \cite{PollackG2}, an expression for this integral was found. To state that result, we define the function
\begin{equation}\label{eq:J' defn}
J'(s) = |q(v_E)|^{-s} \int_{V^*}{|q(v)|^{s} e^{-|\langle v, r_0(i)\rangle|^2}\, \frac{dV}{|q(v)|}}.
\end{equation}
Here, $V^*$ is the $GL_2(\mathbb R)$-orbit that consists of the binary cubics that split over $\R$ and $dV$ denotes the Haar measure on $V^*$. Note that $V^*$ is a subset of $W$, which is the space of binary cubic forms. Also, $v_{E} = (0,\frac{1}{3},\frac{1}{3},0)$ corresponds to the binary cubic $x^2y+xy^2$, and $r_0(i) = (1,-i,-1,i)$ corresponds to the binary cubic $(x-iy)^3$. The quartic form $q$ and the symplectic pairing $\langle \,,\,\rangle$ are as defined in Section \ref{subsec:binary cubics}.


Our first step in computing $ I(s;\ell)$ is proving the following result.

\begin{proposition} \label{prop:I s ell and J'(s)}
We have
\[
I(s;\ell)
\sim
\pi^{-s}  \frac{\Gamma(s+2\ell-3)\Gamma(s+\ell-2)\Gamma((s+\ell-3)/2)^2}{\Gamma((s+\ell)/2)\Gamma(s+\ell-3)\Gamma((3s+3\ell-7)/2)} J'\left(\frac{s+\ell-2}{2}\right),
\]
where the function $J'$ is as defined in \eqref{eq:J' defn}.
\end{proposition}

\begin{proof}
Let $\chi'$ denote the archimedean part of the character $\psi(\langle v_E,\overline{n}\rangle)$, so that $\chi'(n) = e^{2\pi i \langle v_{E},\overline{n}\rangle}$. In the notation of \cite{PollackG2}, compared with the third displayed equation on page 30 of \cite{PollackG2}, we have
\[
I(s;\ell) = \int_{\GL_2(\R)}\int_{(N^{0,E}\backslash N)(\R)}{\frac{|\det(m)|^{-3} e^{2\pi i \langle v_{E},\overline{n}\rangle}}{||x(n,m)||^{(s+\ell)}} \{\text{pr}_K(x(n,m))^{\ell},\mathcal{W}_{\chi'}(m)\}_{K}\,dn\,dm}.
\]
Here, we need to note that in \emph{loc cit}, the function
\[
I(s,\Phi)= \int_{N^{0,E}(\R)\backslash GL_1(\R)\times G_2(\R)}
|t|^s \{\Phi(tg^{-1} \widetilde{v_E}), \mathcal{W}_\chi(g)\}_K\,dg
\]
is used instead. The $\Gamma((s+\ell)/2)$ in the expression for $I(s, \Phi)$ in \emph{loc cit} has disappeared here since $I(s;\ell)$ is defined in terms of the flat section $f_{\ell}(\gamma_0 g, s)$ whereas $I(s,\Phi)$ in \cite{PollackG2} was defined in terms of a section that takes the value $\Gamma((s+\ell)/2)$ at $g=1$.

Now, by following the same argument as in \cite[p. 30]{PollackG2}, we obtain
\[
I(s;\ell) \sim
\int_{\GL_2(\R) \times (N^{0,E}\backslash N)(\R)}{ \frac{ |\det(m)|^{s+\ell-2} e^{2\pi i \beta}}{(|\alpha|^2 + |\beta|^2)^{(s+\ell)/2}} \bigg(\sum_{j}{\binom{\ell}{j} (i\beta)^{\ell-j} |\alpha|^{j} K_0^{(j)}(2\pi |\alpha|)}\bigg)\,dn\,dm}.
\]
Using the change of variables $\beta \mapsto (2\pi)^{-1}\beta$ and $m \mapsto (2\pi 1_2)^{-1}m$, as $\alpha$ depends on $m$, we find that
\[
I(s;\ell) \sim
(2\pi)^{-s} \int_{\GL_2(\R) \times (N^{0,E}\backslash N)(\R)}{ \frac{ |\det(m)|^{s+\ell-2} e^{ i \beta}}{(|\alpha|^2 + |\beta|^2)^{(s+\ell)/2}} \bigg(\sum_{j}{\binom{\ell}{j} (i\beta)^{\ell-j} |\alpha|^{j} K_0^{(j)}( |\alpha|)}\bigg)\,dn\,dm}.
\]
Consequently, $I(s;\ell) \sim (2\pi)^{-s} \Gamma((s+\ell)/2)^{-1} I(s,\Phi)$. The result now follows from the first part of Theorem 6.2 in \cite{PollackG2}.
\end{proof}

\begin{remark}
In~\cite{PollackG2}, the factor $|q(v_E)|^{-s}$ was mistakenly omitted in the first part of Theorem 6.2. It should first appear in the fifth displayed equation on page 32, as $|\det(g)|^2 = |q(v_E)|^{-1}|q(v)|$, and then be carried over to the expression for $I(s,\Phi)$ in the first part of Theorem 6.2.
\end{remark}

As our next step, we now prove


\begin{proposition}\label{prop:J'int}
Let $J'(s)$ be as in \eqref{eq:J' defn}. We have
	\[
	J'(s) \sim
	2^{-6s} \Gamma(2s) \frac{\Gamma(3s-1/2)}{\Gamma(s+1/2)^3}.
	\]
\end{proposition}

\begin{proof}
To compute $J'(s)$, it suffices to integrate over those elements of $V^*$ which have nonzero leading coefficients since the set of such elements of $V^*$ has co-measure zero. Such a binary cubic can be written as
\[
t(w-r_1z)(w-r_2z)(w-r_3z)
\]
for $t, r_1, r_2, r_3 \in \R$.  To compute the integral $J'(s)$, we make the variable change
	\begin{itemize}
		\item $a = t$,
		\item $b = -t(r_1+r_2+r_3)$,
		\item $c=t(r_1 r_2 + r_2 r_3 + r_3 r_1)$,
		\item $d = - t r_1 r_2 r_3$.
	\end{itemize}
The Jacobian of this transformation equals
	\[
	\frac{\partial(a,b,c,d)}{\partial(t,r_1,r_2,r_3)} = \left|\begin{array}{cccc} 1&0 &0 &0 \\ *_1&t&t&t \\ *_2&t(r_2+r_3) & t(r_1+r_3) & t(r_1 + r_2) \\ *_3& tr_2 r_3 &tr_1 r_3 & tr_1 r_2 \end{array}\right|
	= \pm t^3 (r_1-r_2)(r_2-r_3)(r_3-r_1),
	\]
where $*_1, *_2, *_3$ denote some real numbers. Note that we have
	\[
	q(w^2z+wz^2)^{-1} q(t(w-r_1z)(w-r_2z)(w-r_3z))
	= t^4 (r_1-r_2)^2(r_2-r_3)^2(r_3-r_1)^2.
	\]
By combining this with the change of variables, we write
\begin{align*}
J'(s)
=&\int_{t,r_1,r_2,r_3} t^{4s-4} e^{-t^2(1+r_1^2)(1+r_2^2)(1+r_3^2)}
\prod_{1\leq i < j \leq 3}|r_i - r_j|^{2s-2}
\frac{\partial(a,b,c,d)}{\partial(t,r_1,r_2,r_3)} d(t,r_1,r_2,r_3) \\
=& \int_{t,r_1,r_2,r_3}t^{4s-1}
 e^{-t^2(1+r_1^2)(1+r_2^2)(1+r_3^2)}
 \prod_{1\leq i < j \leq 3}|r_i - r_j|^{2s-1}
 \,d(t,r_1,r_2,r_3) \\
 \sim
 & \, \Gamma(2s) \int_{r_1,r_2,r_3}{|1+r_1^2|^{-2s}|1+r_2^2|^{-2s}|1+r_3^2|^{-2s} \prod_{1\leq i < j \leq 3}|r_i - r_j|^{2s-1} \,d(r_1,r_2,r_3)}.
\end{align*}
The integral on the last line is a special case of the Selberg integral. From (1.19) in \cite{Forrester}  with $\alpha = \beta = 2s$, $\gamma = s-\frac12$ and $n = 3$, it follows that
\begin{align*}
 J'(s)
 & \sim 2^{-6s}\, \Gamma(2s) \prod_{j=0}^{2} \frac{\Gamma(4s-1-(2+j)(s-1/2))\Gamma(1+(j+1)(s-1/2))}{\Gamma(2s-j(s-1/2))^2 \Gamma(s+1/2)} \\
 &= 2^{-6s}\, \Gamma(2s) \frac{\Gamma(3s-1/2)}{\Gamma(s+1/2)^3}.
\end{align*}
\end{proof}


\begin{proof}[Proof of Theorem \ref{thm:I*int}]
It immediately follows from Proposition \ref{prop:J'int} that
\[
J'\left(\frac{s+\ell-2}{2}\right)
\sim
2^{-3s-3\ell+6} \Gamma(s+\ell-2) \frac{\Gamma((3s+3\ell-7)/2)}
{\Gamma((s+\ell-1)/2)^3}.
\]
By combining this with Proposition \ref{prop:I s ell and J'(s)}, we obtain
\[
I(s;\ell)
\sim (8\pi)^{-s}\, \frac{\Gamma(s+2\ell-3) \Gamma(s+\ell-2)^2 \Gamma((s+\ell-3)/2)^2}{\Gamma(s+\ell-3)\Gamma((s+\ell-1)/2)^3 \Gamma((s+\ell)/2)}.
\]
Then by \eqref{eq:norm zeta integral},
\begin{align*}
I^*(s;\ell)
&=  2^s \Gamma_\R(s-1) \Gamma_\C(s+\ell-1)\Gamma_\C(s+\ell-2)  I(s;\ell) \\
&\sim
(4\pi)^{-s}
\Gamma_\R(s-1)\Gamma_\C(s+\ell-1)\Gamma_\C(s+\ell-2)
 \\
&\quad  \times \frac{\Gamma(s+2\ell-3)\Gamma(s+\ell-2)}{\Gamma((s+\ell-1)/2)\Gamma((s+\ell)/2)} \frac{\Gamma(s+\ell-2) \Gamma((s+\ell-3)/2)^2}{\Gamma(s+\ell-3)\Gamma((s+\ell-1)/2)^2}.
\end{align*}
Further by using the duplication formula
\[
\Gamma(2z)=2^{2z-1}\pi^{-1/2}\Gamma(z) \Gamma\Big(z+\tfrac12\Big),
\]
we find that
\begin{align*}
I^*(s;\ell) &\sim \Gamma_\R(s-1)\Gamma_\C(s+\ell-1)\Gamma_\C(s+\ell-2)
 \\
&\quad  \times \frac{\Gamma_\C(s+2\ell-3)\Gamma_\C(s+\ell-2)}{\Gamma_\R(s+\ell-1)\Gamma_\R(s+\ell)} \frac{\Gamma_\C(s+\ell-2) \Gamma_\R(s+\ell-3)^2}{\Gamma_\C(s+\ell-3)\Gamma_\R(s+\ell-1)^2}\\
&=\Gamma_\R(s-1)\Gamma_\C(s+\ell-3)\Gamma_\C(s+\ell-2)\Gamma_\C(s+2\ell-3).
\end{align*}
This completes the proof.
\end{proof}

\section{Proofs of the Main Results}\label{sec: proofs of main results}
In this short section, we combine our result on the archimedean integral $I^*(s;\ell)$ with our results on other local integrals $I_p(s)$ and complete the proofs of our main results. Consider
\[
I_\ell(\varphi,s) = \int_{G_2(\Q)\backslash G_2(\A)}{\{\varphi(g),E_{\ell}^*(g,s)\}_K\,dg}
\]
as in Section \ref{thm:intro}. We have

\begin{theorem}\label{thm:globInt}
The integral $I_{\ell}(\varphi,s)$ is equal to $a_{\varphi}(\Z^3) \Lambda(\pi,\operatorname{Std},s-2)$, up to a nonzero constant.
\end{theorem}
\begin{proof}
Note that
\[
E_{\ell}^*(g,s) = 2^s \zeta(s-1)^2\zeta(2s-4)  \Gamma_\R(s-1) \Gamma_\C(s+\ell-1)\Gamma_\C(s+\ell-2) E(g,\Phi_f,s),
\]
where $E(g,\Phi_f,s)$ is as defined in \eqref{def:E g phif s}. Also by Theorem \ref{thm:I*int},
	\[
	I^*(s;\ell)
	\sim
	\Gamma_\R(s-1)\Gamma_\C(s+\ell-3)\Gamma_\C(s+\ell-2)\Gamma_\C(s+2\ell-3).
	\]
Taking into account the normalization of the Eisenstein series $E^*_{\ell}(g,s)$, the theorem follows directly from Theorem \ref{thm:Ip} and Theorem \ref{thm:I*int} by using the technique of ``non-unique models", also known as ``new-way (Eulerian) integral", which is explained in \cite{BumpGinzburg}, \cite{GurevichSegal} and \cite{PSRallis}.

\end{proof}

\begin{proof}[Proof of Theorem \ref{thm:intro}]
In Theorem \ref{thm:Eis}, we proved that  $E_{\ell}^*(g,s) = E_{\ell}^*(g,5-s)$. By combining this functional equation with Theorem~\ref{thm:globInt}, the result follows.

\end{proof}

\begin{proof}[Proof of Corollary \ref{cor:introDS}]
This follows from Theorem \ref{thm:Ip} exactly as in Section 5.9 of \cite{PollackG2}.
\end{proof}

\bibliographystyle{plain}
\bibliography{G2L}
\end{document}